\UseRawInputEncoding
\documentclass{amsart}

\usepackage{amsmath, amssymb, amsthm,enumitem,bm,pict2e,xcolor,hyperref,url}
\usepackage{indentfirst}

\hypersetup{
    colorlinks=true,
    linkcolor=blue,
    filecolor=magenta,      
    urlcolor=cyan,
}

\theoremstyle{definition}
\newtheorem{theorem}{Theorem}
\newtheorem*{theorem*}{Theorem}
\numberwithin{theorem}{section}
\newtheorem{proposition}[theorem]{Proposition}
\newtheorem{lemma}[theorem]{Lemma}
\newtheorem{remark}[theorem]{Remark}
\newtheorem{example}[theorem]{Example}
\newtheorem{cor}[theorem]{Corollary}
\newtheorem*{question*}{Question}

\newtheorem{innercustomthm}{Theorem}
\newenvironment{customthm}[1]
  {\renewcommand\theinnercustomthm{#1}\innercustomthm}
  {\endinnercustomthm}
\DeclareMathOperator*{\bigtimes}{\text{\raisebox{-0.25 ex}{\huge$\times$}}}

\DeclareMathOperator{\grad}{grad}
\DeclareMathOperator{\wt}{wt}

\DeclareMathOperator{\Sym}{Sym}

\DeclareMathOperator{\gr}{gr}
\DeclareMathOperator{\initial}{in}
\DeclareMathOperator{\sq}{sq}
\DeclareMathOperator{\essential}{es}

\newcommand{\la}{\lambda}
\newcommand{\om}{\omega}
\newcommand{\mU}{\mathcal U}
\newcommand{\fg}{\mathfrak{g}}
\newcommand{\msl}{\mathfrak{sl}}
\newcommand{\fn}{\mathfrak{n}}
\newcommand{\fh}{\mathfrak{h}}
\newcommand{\bC}{\mathbb{C}}
\newcommand{\bP}{\mathbb{P}}
\newcommand{\bR}{\mathbb{R}}
\newcommand{\bZ}{\mathbb{Z}}
\newcommand{\bd}{{\bf d}}

\mathcode`l="8000
\begingroup
\makeatletter
\lccode`\~=`\l
\DeclareMathSymbol{\lsb@l}{\mathalpha}{letters}{`l}
\lowercase{\gdef~{\ifnum\the\mathgroup=\m@ne \ell \else \lsb@l \fi}}%
\endgroup

\title{Gelfand--Tsetlin degenerations of representations and flag varieties}
\author{I. Makhlin}
\address{Skolkovo Institute of Science and Technology\\ 
Center for Advanced Studies\\
Bolshoy Boulevard 30, bld. 1\\
Moscow 121205\\
Russia}%\newline
\address{National Research University Higher School of Economics\\
International Laboratory of Representation Theory and Mathematical Physics\\
Ulitsa Usacheva 6\\Moscow 119048\\Russia\vspace{2ex}}
\email{imakhlin@mail.ru}

\begin{document}

\maketitle

\begin{abstract}
Our main goal is to show that the Gelfand--Tsetlin toric degeneration of the type A flag variety can be obtained within a degenerate representation-theoretic framework similar to the theory of PBW degenerations. In fact, we provide such frameworks for all Gr\"obner degenerations intermediate between the flag variety and the GT toric variety. These degenerations are shown to induce filtrations on the irreducible representations and the associated graded spaces are acted upon by a certain associative algebra. To achieve our goal, we construct embeddings of our Gr\"obner degenerations into the projectivizations of said associated graded spaces in terms of this action. We also obtain an explicit description of the maximal cone in the tropical flag variety that parametrizes the Gr\"obner degenerations we consider. In an addendum we propose an alternative solution to the problem which relies on filtrations and gradings by non-abelian ordered semigroups. 
\end{abstract}

\section*{Introduction}

Toric degenerations as well as other flat degenerations of flag varieties have been studied extensively over the past thirty years. A fascinating aspect of this field is the large number of constructions from diverse research areas that produce these degenerations as well as the relationships between these constructions. We recommend~\cite{knutson} for a short overview and~\cite{fafltoric} for a more detailed survey. 

One such source of degenerations of flag varieties is the recently popular field of PBW degenerations where one first degenerates the representations and then defines the degenerate variety in this context. The original construction here is due to~\cite{Fe2}. This paper was written with the initial goal of showing that the scope of this approach is broader than previously believed and that other, seemingly unrelated, degenerations of the flag variety can be obtained in similar fashion. More specifically, we focus our attention on what is one of the earliest and most popular flat degenerations: the Gelfand--Tsetlin toric degeneration of the type A flag variety. It was first constructed in~\cite{GL}, the connection to Gelfand--Tsetlin polytopes is due to~\cite{KM}. The reader is also referred to~\cite[Section 14]{MS} for an exposition and to~\cite{BCKS,Ch,Ca,NNU,L} for various generalizations and interpretations. The main goal of this paper is to answer the following.\vspace{2mm}
\centerline{
\begin{minipage}{.85\textwidth}
\emph{\textbf{Motivating question.} How can one degenerate the representation theory of $\msl_n$ in order to obtain the Gelfand--Tsetlin toric degeneration within a context similar to that of PBW degenerations ?}
\end{minipage}
}\vspace{2mm}

First, let us outline a general setting which covers the known constructions of PBW degenerations. We will refer to such settings as ``degenerate representation theories''. For a complex semisimple Lie algebra $\fg$ choose a triangular decomposition and let $\fn_-$ be the nilpotent subalgebra spanned by negative root vectors. Consider a filtration of the universal enveloping algebra $\mU(\fn_-)$ which respects multiplication and provides an associated graded algebra $\gr\mathcal U(\fn_-)$. The filtration on $\mU(\fn_-)$ induces a filtration on every finite-dimensional irreducible representation $L_\la$, the associated graded space $\gr L_\la$ is then naturally a $\gr\mathcal U(\fn_-)$-module, the degeneration of $L_\la$. Furthermore, for certain well-behaved filtrations it turns out that $\gr\mathcal U(\fn_-)$ is itself a universal enveloping algebra, hence $\gr L_\la$ is acted upon by the corresponding Lie group $N^{\gr}$. In $\bP(\gr L_\la)$ consider the point corresponding to the line of highest weight vectors. The closure $F^{\gr}$ of the orbit of this point under the $N^{\gr}$-action is the corresponding degenerate flag variety which, in some situations, constitutes a flat degeneration.

We point out that most of the research carried out on this subject is concerned specifically with the standard filtration by PBW degree. In this case $\gr\mathcal U(\fn_-)$ is the symmetric algebra $S(\fn_-)$ and $N^{\gr}$ is simply $\bC^{\dim\fn_-}$ under addition, while $F^{\gr}$ is, in general, some non-toric singular variety known as the \emph{abelian PBW degeneration}. The (ongoing) studies of this particular setting have resulted in works by a wide range of authors, 
%reaching into representation theory% (\cite{Fe1}, \cite{FFL1}, \cite{FFL2}, \cite{CF}, ...)
%, algebraic geometry % (\cite{Fe2}, \cite{CFR}, \cite{FFiL}, \cite{Ki}, ...)
%and combinatorics. % (\cite{ABS}, \cite{chainorder}, \cite{FM}, \cite{Bi}, ...)
in particular, see ~\cite{Fe2,FFL1,ABS,FFL2,CFR,CL}. 
However, to us it is important that this scenario can be generalized beyond the standard PBW filtration to obtain a wider class of flat degenerations~\cite{favourable,fafefom}. 

We give two different answers to our motivating question which bear certain parallels to these two respective papers. However, each of these answers involves a generalization of the above concept of a degenerate representation theory.

To give the first answer we consider a family of Gr\"obner degenerations $F^S$ that includes the Gelfand--Tsetlin toric degeneration (from now on we set $\fg=\msl_n$). %These $F^S$ are given by $\bZ$-gradings $S$ on the Pl\"ucker variables satisfying certain linear inequalities.
We show that each such Gr\"obner degeneration $F^S$ defines a filtration on every $L_\la$ which produces the degenerate representations $L_\la^S$. These filtrations may also be obtained representation-theoretically, however, the key difference from previous constructions is that we first define a certain deformation $\Phi_n$ of $\mathcal U(\fn_-)$ that acts on every $\fg$-representation $L_\la$ and degenerate the action of this algebra rather than that of $\mU(\fn_-)$. %The algebra $\Phi_n$ is generated by $\varphi_{i,j}$ with $1\le i<j\le n$ (i.e. corresponding to the root vectors in $\fn_-$) and is filtered by assigning $\bZ$-degrees to these $\varphi_{i,j}$. 
Thus we obtain a $\Phi_n$ action on every $L_\la^S$. However, $\Phi_n$ is not a universal enveloping algebra and we need to find a different method of recovering the degenerate flag variety from of the degenerate representation theory, a method that would not rely on the existence of a group action. We propose two such methods within our first answer to the motivating question. The first is to imitate the group action by taking exponentials. Let $\Phi_n$ be generated by $\varphi_{i,j}$ with $1\le i<j\le n$, corresponding to the negative roots.
\begin{customthm}{A}[cf. Theorem~\ref{main}]\label{thma}
$F^S$ is isomorphic to the closure of the set of points of the form $\prod_{i,j}\exp(c_{i,j}\varphi_{i,j})\bm v_\la^S$ where $c_{i,j}\in\bC$ and $\bm v_\la^S\in\bP(L_\la^S)$ is the highest weight line.
\end{customthm}

The second method relies on the existence of tensor products and Cartan components for our degenerate representations. These provide us with a commutative algebra structure on $\mathcal P^S=\bigoplus_\la (L_\la^S)^*$. The degenerate flag variety can be defined as the ``multigraded $\mathrm{Proj}$'' of $\mathcal P^S$, in other words, we prove the following
\begin{customthm}{B}[cf. Theorem~\ref{mainproj}]\label{thmb}
$\mathcal P^S$ is generated by the components $(L_{\om_i}^S)^*$ (for fundamental weights $\om_i$) and the kernel of the surjection $\bC[\bigoplus_i (L_{\om_i}^S)^*]\to \mathcal P^S$ cuts out a subvariety in $\bigtimes_i\bP(L_{\om_i}^S)$ that is isomorphic to $F^S$.
\end{customthm}
While the characterization in Theorem~\ref{thma} is similar to the original orbit closure definition, the one in Theorem~\ref{thmb} has the advantage of being independent of the highest weight. Let us also note that while analogs of Theorem~\ref{thmb} are not given explicitly in~\cite{FFL1,favourable,fafefom} such analogs do hold and can easily be derived. 

This first answer to the motivating question has several similarities to~\cite{fafefom}. In both cases the Gr\"obner degenerations in consideration are parametrized by a polyhedral cone and points in the cone's interior provide a toric degeneration. In~\cite{fafefom} this toric degeneration is the one associated with the Feigin--Fourier--Littelmann--Vinberg (or FFLV) polytope (\cite{FFL1}), here we obtain the desired Gelfand--Tsetlin toric degeneration. Moreover, similarly to~\cite{fafefom} we show that this cone can be identified with a maximal cone  in the tropical flag variety. Thus we obtain what is the second explicit description of a maximal cone in every (type A) tropical flag variety. Furthermore, an important role in this paper is played by a monomial basis in the spaces $L_\la$ and $L_\la^S$, this basis is given by integer points in a unimodular transform of the Gelfand--Tsetlin polytope. Its role is similar to that of the FFLV basis in~\cite{fafefom}, which is a monomial basis given by integer points in the FFLV polytope. We point out that the basis obtained here should not be viewed as an entirely new construction, it is seen to be a subset of the Chari--Loktev basis in the corresponding local Weyl module (\cite{ChL}). Its existence also follows from the results in~\cite{R} concerning canonical bases in quantum groups (although the connection with Gelfand--Tsetlin polytopes is not mentioned).

The second alternative answer to our motivating question is a later result that is presented as an addendum. The approach here has similarities to~\cite{favourable}. In the latter paper a PBW basis in $\mU(\fn_-)$ is fixed, such a basis is in natural bijection with the abelian monoid $\bZ_{\ge0}^{\{1\le i<j\le n\}}$. An additive total order on this monoid is then chosen, this provides a $\bZ_{\ge0}^{\{1\le i<j\le n\}}$-filtration on $\mU(\fn_-)$. Under certain restrictions on the total order this filtration is multiplicative and the associated graded algebra is a universal enveloping algebra. The authors obtain a degenerate representation theory of the kind described above where all $L_\la^{\gr}$ are $\bZ_{\ge0}^{\{1\le i<j\le n\}}$-graded. The orbit closure is then seen to be a toric variety. A specific choice of the PBW basis and total order is shown to provide the FFLV toric degeneration.

The idea here is to proceed similarly to~\cite{favourable} without, however, having to fix a PBW basis. Note that the set of \emph{all} PBW monomials in $\mU(\fn_-)$ is in bijection with a \emph{non-abelian} free monoid $\mathcal F$ with $n\choose 2$ generators. Any multiplicative total order on $\mathcal F$ provides a multiplicative $\mathcal F$-filtration on $\mU(\fn_-)$. We then may consider the corresponding associated $\mathcal F$-graded algebra and associated $\mathcal F$-graded representations. The key difference from degenerate representation theories that were considered previously is that these objects are graded by a non-abelian semigroup while in previous constructions the grading was always by $\bZ$ or $\bZ^k$. 

We then show how the Gelfand--Tsetlin toric degeneration can be obtained within this general framework. We define a specific total order on $\mathcal F$ and describe the resulting degenerate representation theory. Next we prove an analog of Theorem~\ref{thma} where the $\varphi_{i,j}$ are replaced by the generators of $\gr\mU(\fn_-)$ and $F^S$ is specifically the toric degeneration. We also define a tensor product on the degenerate representations and prove an analog of Theorem~\ref{thmb}. 

The main advantage of this second answer to our question is that the degenerate algebra action appears naturally as a degeneration of the $\mU(\fn_-)$-action rather than by degenerating a different action defined ad hoc, as was the case with $\Phi_n$. A disadvantage, however, is that the other (non-toric) Gr\"obner degenerations appearing in the first answer are not recovered, nor is the connection with tropical geometry.

\section{Generalities on Gr\"obner degenerations of flag varieties}\label{generalities}

For a fixed $n\ge 2$ consider the Lie group $G=SL_n(\bC)$ with Borel subgroup $B$ and tangent algebra $\fg=\msl_n(\bC)$. Let $\mathfrak b\subset\fg$ be the Borel subalgebra tangent to $B$, let $\fh\subset\mathfrak b$ be the Cartan subalgebra and let $\fg=\mathfrak b\oplus\fn_-$ for nilpotent subalgebra $\fn_-$. For $1\le k\le n-1$ denote the simple roots $\alpha_k\in\fh^*$  and let $\om_k\in\fh^*$ be the corresponding fundamental weights. Denote the positive roots \[\alpha_{i,j}=\alpha_i+\ldots+\alpha_{j-1}\] for $1\le i<j\le n-1$. Let $\fn_-$ be spanned by negative root vectors $f_{i,j}$ with weight $-\alpha_{i,j}$. Our choice of basis in $\fh^*$ will be the set of fundamental weights, i.e.\ $(a_1,\ldots,a_{n-1})$ will denote the weight $a_1\om_1+\ldots+a_{n-1}\om_{n-1}$.

For a dominant integral $\fg$-weight $\la$ let $L_\la$ be the irreducible representation of $\fg$ with highest weight $\la$ and highest weight vector $v_\la$. Let the $n$-dimensional complex space $V$ be the tautological representation of $\fg=\msl_n$ with basis $e_1,\ldots,e_n$. The irreducible representations with fundamental highest weights can be explicitly described as $L_{\om_k}=\wedge^k V$ with a basis consisting of the vectors \[e_{i_1,\ldots,i_k}=e_{i_1}\wedge\ldots\wedge e_{i_k}.\] We may assume that $v_{\om_k}=e_{1,\ldots,k}$.

Consider the variety of complete flags $F=G/B$ and the Pl\"ucker embedding \[F\subset\bP=\bP(L_{\om_1})\times\ldots\times\bP(L_{\om_{n-1}}).\] The product $\bP$ is equipped with the Pl\"ucker coordinates $X_{i_1,\ldots,i_k}$ with $1\le k\le n-1, 1\le i_1<\ldots<i_k\le n$, coordinate $X_{i_1,\ldots,i_k}$ corresponding to $e_{i_1,\ldots,i_k}\in L_{\om_k}$. The homogeneous coordinate ring of $\bP$ is $R=\bC[\{X_{i_1,\ldots,i_k}\}]$. The multi-homogeneous coordinate ring of $F$ is the \emph{Pl\"ucker algebra} $\mathcal P=R/I$, where $I$ is the ideal of Pl\"ucker relations.

Note that $R$ is naturally graded by the semigroup of dominant integral weights with the homogeneous component $R_\la$ corresponding to weight $\la=(a_1,\ldots,a_{n-1})$ spanned by monomials with total degree in variables of the form $X_{i_1,\ldots,i_k}$ equal to $a_k$. We will denote this grading $\wt$. Since the ideal $I$ is $\wt$-homogeneous, so is $\mathcal P$. In the latter, the homogeneous component $\mathcal P_\la$ of degree $\la$ is identified with the dual representation $L_\la^*$. These classical definitions and results concerning $\msl_n$-representations and flag varieties can be found in~\cite{carter} and~\cite[Chapter 9]{fulton}.

Now consider a collection of integers $S=(s_{i_1,\ldots,i_k})$, one for each Pl\"ucker variable. This provides a $\bZ$-grading on $R$ by setting $\grad^S X_{i_1,\ldots,i_k}=s_{i_1,\ldots,i_k}$. Consider the initial ideal $\initial_{\grad^S}I$ (spanned by nonzero components of minimal grading of elements of $I$). We will be considering the subvariety $F^S$ in $\bP$ defined by this ideal. Varieties of the form $F^S$ are known as {\it Gr\"obner degenerations} of $F$.

We have a decreasing $\bZ$-filtration on $R$ with the $m$th filtration component $R_{\ge m}$ being spanned by monomials in $R$ of $\grad^S$ no less than $m$. This induces a decreasing $\bZ$-filtration on $\mathcal P$ with components $\mathcal P_{\ge m}$, note that this is a filtered algebra structure. We denote the associated $\bZ$-graded algebra $\mathcal P^S=\bigoplus_m \mathcal P_{\ge m-1}/\mathcal P_{\ge m}$. 

\begin{proposition}\label{initdegen}
$\mathcal P^S$ and $R/\initial_{\grad^S}I$ are isomorphic as $\bZ$-graded algebras. 
\end{proposition}
\begin{proof}
The associated graded algebra of $R$ with respect to the filtration $(\cdot)_{\ge m}$ is again $R$ with the same grading $\grad^S$. This associated graded algebra, however, projects naturally onto $\mathcal P^S$. We obtain a surjection of $\bZ$-graded algebras from $R$ onto $\mathcal P^S$ and are left to show that the kernel of this surjection is $\initial_{\grad^S}I$. 

A $\grad^S$-homogeneous element $p\in R$ lies in this kernel if and only if the image of $p$ in $\mathcal P$ is contained in $\mathcal P_{\ge\grad^S(p)+1}$. This happens if and only if for some $q\in I$ we have $p+q\in R_{\ge \grad^S(p)+1}$. The latter condition is equivalent to $p$ being the initial part of $-q$ and we see that the kernel is spanned by the initial parts of elements of $I$.
\end{proof}

In particular, we obtain isomorphisms between the $\wt$-homogeneous components, i.e.\ we have identified every $\wt$-homogeneous component $\mathcal P^S_\la$ of a Gr\"obner degeneration of the Pl\"ucker algebra with a certain associated graded space of the dual irreducible representation.

\begin{remark}
Isomorphisms between quotients by initial ideals and associated graded rings are a very general phenomenon. A special case important for algebraic geometry is the filtration by the powers of an ideal (see~\cite[Proposition 15.28]{E}), for a setting closer to ours see~\cite[Lemma 3.4]{KaM}. This phenomenon will also reappear in this paper as Proposition~\ref{adinitdegen}. However, somewhat surprisingly, the author was not able to find this proved in the literature as a general fact that would directly imply either of the two propositions, see also MathOverflow question~\cite{MO}.
\end{remark}

Now, for an integral dominant weight $\la=(a_1,\ldots,a_{n-1})$ consider the tensor product \[U_\la=L_{\omega_1}^{\otimes a_1}\otimes\ldots\otimes L_{\omega_{n-1}}^{\otimes a_{n-1}}.\] The subrepresentation of $U_\la$ generated by the highest weight vector \[u_\la=v_{\omega_1}^{\otimes a_1}\otimes\ldots\otimes v_{\omega_{n-1}}^{\otimes a_{n-1}}\] is the irreducible representation $L_\la$ (naturally dual to $\mathcal P_\la$).

If we grade $L_{\om_k}$ by setting $\grad^S e_{i_1,\ldots,i_k}=s_{i_1,\ldots,i_k}$, a grading (which we also denote $\grad^S$) on $U_\la$ is induced. We may consider an increasing $\bZ$-filtration on $U_\la$ with the $m$th component $(U_\la)_{\le m}$ being spanned by $\grad^S$-homogeneous elements of $\grad^S$ no greater than $m$. This induces a $\bZ$-filtration on $L_\la\subset U_\la$ with components $(L_\la)_m$, denote the associated graded space $L_\la^S$ with homogeneous components $(L_\la^S)_m$. %For every tensor product of the vectors $e_{i_1,\ldots,i_k}$ in $U_\la$ we have the corresponding product of the variables $X_{i_1,\ldots,i_k}$ in $R_\la$. Since the corresponding set of products is a basis in each of $U_\la$ and $R_\la$, we can view the two spaces as mutually dual. This duality identifies the gradings $\grad^S$, hence the abuse of notations.
\begin{proposition}\label{dualspaces}
 $L_\la^S$ and $\mathcal P_\la^S$ are dual as $\bZ$-graded vector spaces.
\end{proposition}
\begin{proof}
Let us consider the subrepresentation \[W_\la=\Sym^{a_1}(L_{\om_1})\otimes\ldots\otimes\Sym^{a_{n-1}}(L_{\om_{n-1}})\subset U_\la,\] note that $W_\la$ is a graded subspace and $L_\la\subset W_\la$. The space $W_\la$ is dual to $R_\la$ where the symmetrization of a tensor product of some vectors $e_{i_1,\ldots,i_k}$ is the basis element dual to the product of the corresponding variables $X_{i_1,\ldots,i_k}$. The increasing $\bZ$-filtration on $W_\la$ (given by $(W_\la)_{\le m}=(U_\la)_{\le m}\cap W_\la$) is dual to the decreasing $\bZ$-filtration on $R_\la$ in the sense that the subspace $(W_\la)_{\le m}$ is dual to the subspace $(R_\la)_{\ge m+1}$. This provides a duality between their associated graded spaces which are again $W_\la$ and $R_\la$ with the same gradings. The space $L_\la^S$ is, by definition, embedded into the associated graded space $W_\la$ and the space $\mathcal P_\la^S$ is (as noted in the proof of Proposition~\ref{initdegen}) a projection of the associated graded space $R_\la$. We are to show that the kernel $\initial_{\grad^S} I_\la$ of the latter projection is the orthogonal of the subspace $L_\la^S\subset W_\la$. 

Now, it is known (see, for instance, Theorem 8.2.2 and Proposition 9.1.1 in~\cite{fulton}) that $L_\la\subset W_\la$ is the orthogonal of the kernel $I_\la$ of the projection of $R_\la$ onto $\mathcal P_\la$. However, if we consider an element of $v\in L_\la\subset W_\la$ and take its projection onto the $\grad^S$-homogeneous component of maximal grading for which the projection is nonzero we will obtain an element of $L_\la^S\subset W_\la$ and $L_\la^S$ is spanned by elements of this form. By definition $\initial_{\grad^S} I_\la$ is spanned by the $\grad^S$-initial parts of elements of $I_\la$. One sees that such an initial part annihilates the mentioned projection in $L_\la^S$ and the orthogonality follows.
\end{proof}

As discussed above, we have an embedding $L_\la^S\subset W_\la\subset U_\la$. We also have the Segre embedding $$\mathbb P(L_{\om_1})^{a_1}\times\ldots\times\mathbb P(L_{\om_{n-1}})^{a_{n-1}}\subset\mathbb P(U_\la)$$ and, for regular $\la$ (i.e.\ all $a_k>0$), the embedding $$\mathbb P\subset \mathbb P(L_{\om_1})^{a_1}\times\ldots\times\mathbb P(L_{\om_{n-1}})^{a_{n-1}}$$ where $\mathbb P(L_{\om_k})$ is embedded diagonally into $\mathbb P(L_{\om_k})^{a_k}$. We obtain an embedding \[F^S\subset\bP\subset\bP(U_\la).\]
\begin{proposition}\label{FSinPLS}
For a regular $\la$ the image of $F^S$ under this embedding is contained in $\bP(L_\la^S)\subset\bP(U_\la)$.
\end{proposition}
\begin{proof}
The image of the Segre embedding and, therefore, of $F^S$ lies in $\bP(W_\la)$. In view of Proposition~\ref{dualspaces} and its proof, to show that a point in $F^S$ is contained in $\bP(L_\la^S)$ we are to show that the corresponding line in $W_\la$ is annihilated by every element of $\initial_{\grad^S} I_\la\subset R_\la$ (where these elements are viewed as functionals on $W_\la$). This, however, is straightforward from the definitions.
\end{proof}
The above proposition provides an embedding of $F^S$ into $\bP(L_\la^S)$. When the degeneration is trivial, i.e.\ all $s_{i_1,\ldots,i_k}=0$ and $F^S=F$, we obtain the usual embedding of $F$ into $\bP(L_\la)$ as the closure of the orbit of the point corresponding to $\bC v_\la$ under the action of the group $\exp(\fn_-)$.

Proposition~\ref{FSinPLS} is stated only for regular $\la$ since a version of this theorem for singular $\la$ would concern degenerations of partial flag varieties rather than $F$. Sections~\ref{gtdegens} and~\ref{degenact} will contain more results concerned only with complete flag varieties and/or regular highest weights. This is, again, done to avoid overcomplicating the notations in these key sections. However, partial flag varieties and singular highest weights will be discussed in Section~\ref{singular} and the corresponding generalizations of the results will be given there.

Furthermore, all the results found here as well as in the sections below can be formulated for real (rather than integer) gradings. The reason for us to assume that $s_{i_1,\ldots,i_k}\in\bZ$ is that the real case would require us to consider spaces graded by the semigroup generated by the $s_{i_1,\ldots,i_k}$ rather than by $\bZ$. This would complicate the notations with virtually no gain in mathematical merit. The only disadvantage integer gradings give us is not being able to work directly with cones in the Gr\"obner fan in the proof of Theorem~\ref{maxcone}, only with their sets of integer points. This, however, is easily circumvented.

\begin{remark}\label{projlim}
For integral dominant weights $\la$ and $\la'$ such that $\la-\la'$ is also dominant we have a surjection $L_\la\mapsto L_{\la'}$ of $\mU(\fn_-)$-modules. This is a projective system with projective limit $\mU(\fn_-)$. Now, setting $s_{i_1,\ldots,i_{k}}:=s_{i_1,\ldots,i_{k}}-s_{1,\ldots,k}$ for all tuples $i_1,\ldots,i_{k}$ does not change the degeneration $F^S$, therefore we may assume that all $s_{1,\ldots,k}=0$. With this assumption in place one can see that the said projective system respects the $\bZ$-filtrations on the $L_\la$ in the sense that it induces an increasing $\bZ$-filtration on the limit $\mU(\fn_-)$. 

We will recover an explicit description of this filtration on $\mU(\fn_-)$ for the case of Gelfand--Tsetlin degenerations in Theorem~\ref{Lfiltration}. However, it would be interesting to see a more direct definition of this filtration in the general case.
\end{remark}

\section{Polytopes and monomial bases}\label{bases}

Consider $\Theta=\bR^{\{1\le i<j\le n\}}$, for a point $T\in\Theta$ denote its coordinates $T_{i,j}$. For each tuple $1\le i_1<\ldots<i_k\le n$ with $1\le k\le n-1$ we define a vector $T(i_1,\ldots,i_k)\in\Theta$ coordinatewise by setting
\begin{equation}
T(i_1,\ldots,i_k)_{l,m}=
\begin{cases}
1\text{ if }l\le k\text{ and }m=i_l,\\
0\text{ otherwise}. 
\end{cases}
\end{equation}
In other words, we have the coordinate corresponding to pair $(l,i_l)$ equal to $1$ for every $1\le l\le k$ with $i_l>l$ and all other coordinates zero.

For a $1\le k\le n-1$ let $\Pi_{\om_k}\subset\Theta$ be the set of all $T(i_1,\ldots,i_k)$. Next, for an integral dominant weight $\la=(a_1,\ldots,a_{n-1})$ consider the Minkowski sum 
\begin{equation}\label{minksum}
\Pi_\la=\underbrace{\Pi_{\om_1}+\ldots+\Pi_{\om_1}}_{a_1}+\ldots+\underbrace{\Pi_{\om_{n-1}}+\ldots+\Pi_{\om_{n-1}}}_{a_{n-1}}.
\end{equation}
We also introduce a convex lattice polytope $P_\la\subset\Theta$ consisting of points $T$ such that
\begin{enumerate}[label=(\roman*)]
\item $T_{i,j}\ge0$ for all $1\le i<j\le n$;
\item $\sum_{l=j}^n T_{i,l}-\sum_{l=j+1}^n T_{i+1,l}\le a_i$ for all $1\le i<j\le n$.
\end{enumerate}
The second sum in (ii) is empty if $j=n$.

Before we proceed, let us recall the definition of Gelfand--Tsetlin polytopes introduced in~\cite{GT}. For each integral dominant weight $\la=(a_1,\ldots,a_{n-1})$ the corresponding GT polytope $GT_\la$ is a convex lattice polytope in $\Theta$ composed of points $T$ such that
\begin{enumerate}[label=(\roman*)]
\setcounter{enumi}{2}
\item $\la_i\ge T_{i,i+1}\ge\la_{i+1}$ for all $1\le i\le n-1$ where $\la_i=a_i+\ldots+a_{n-1}$ and $\la_n=0$;
\item $T_{i,j-1}\ge T_{i,j}\ge T_{i+1,j}$ for all pairs $1\le i<j\le n$ with $j>i+1$.
\end{enumerate}
Let $\Gamma_\la$ denote the set of integer points in $GT_\la$, elements of $\Gamma_\la$ are known as \emph{Gelfand--Tsetlin patterns}. A key property of GT polytopes established in~\cite{GT} is that $\Gamma_\la$ enumerates a basis in $L_\la$, hence $|\Gamma_\la|=\dim L_\la$. We will also make use of the following {\it Minkowski sum property}.
\begin{proposition}\label{gtminkowski}
For integral dominant weights $\la$ and $\mu$ the Minkowski sum $\Gamma_\la+\Gamma_\mu$ coincides with $\Gamma_{\la+\mu}$.
\end{proposition}
\begin{proof}
A proof can, for instance, be found in~\cite[Theorem 2.5]{chainorder} in the much more general context of marked chain-order polytopes of which the GT polytopes are a special case.
\end{proof}

\begin{lemma}\label{gtequiv}
$\Pi_\la$ is the set of integer points in $P_\la$. Furthermore, $P_\la$ is unimodularly equivalent to the Gelfand--Tsetlin polytope $GT_\la$. 
\end{lemma}
\begin{example}\label{sl3polytopes}
Before we proceed with the proof let us illustrate the definitions and the Lemma with an example. Let $n=3$, we visualize $T\in\Theta$ as ${T_{1,2}\,T_{2,3}}\atop{T_{1,3}}$. We have \[\Pi_{\om_1}=\left\{T(1)={{0\,0}\atop0},T(2)={{1\,0}\atop0},T(3)={{0\,0}\atop1}\right\}\] and \[\Pi_{\om_2}=\left\{T(1,2)={{0\,0}\atop0},T(1,3)={{0\,1}\atop0},T(2,3)={{1\,1}\atop0}\right\}.\] For $\la=\om_1+\om_2$ we obtain \[\Pi_\la=\left\{{{0\,0}\atop0},{{0\,1}\atop0},{{1\,1}\atop0},{{1\,0}\atop0},{{2\,1}\atop0},{{0\,0}\atop1},{{0\,1}\atop1},{{1\,1}\atop1}\right\}.\] We see that $\Pi_{\om_1}$ is the set of integer points in the polytope $P_{\om_1}$ defined by the inequalities $T_{i,j}\ge 0$, $T_{1,2}+T_{1,3}-T_{2,3}\le 1$, $T_{1,3}\le 1$ and $T_{2,3}\le 0$. Polytope $P_{\om_2}$ is given by $T_{i,j}\ge 0$, $T_{1,2}+T_{1,3}-T_{2,3}\le 0$, $T_{1,3}\le 0$ and $T_{2,3}\le 1$ and $P_{\la}$ is given by $T_{i,j}\ge 0$, $T_{1,2}+T_{1,3}-T_{2,3}\le 1$, $T_{1,3}\le 1$ and $T_{2,3}\le 1$. 

Now, we have $(\la_1,\la_2,\la_3)=(2,1,0)$ and the set $\Gamma_\la$ is seen to be \[\left\{{{2\,1}\atop2},{{2\,0}\atop2},{{2\,0}\atop1},{{2\,1}\atop1},{{2\,0}\atop0},{{1\,1}\atop1},{{1\,0}\atop1},{{1\,0}\atop0}\right\}.\] Consider the affine transformation $\psi$ of $\Theta$ given by $\psi:T\mapsto{{2-T_{1,3}\:1-T_{2,3}}\atop{2-T_{1,2}-T_{1,3}}}$. Observe that $\psi(\Pi_\la)=\Gamma_\la$. Moreover, if one takes the inequalities defining $GT_\la$ (i.e.\ $2\ge T_{1,2}\ge 1\ge T_{2,3}\ge 0$ and $T_{1,2}\ge T_{1,3}\ge T_{2,3}$) and substitutes every occurrence of $T_{i,j}$ with the expression $\psi(T)_{i,j}$ defined above, one will end up with the 6 inequalities defining $P_\la$. For instance, $2\ge T_{1,2}$ turns into $2\ge 2-T_{1,3}\Leftrightarrow T_{1,3}\ge 0$ or $T_{1,3}\ge T_{2,3}$ turns into $2-T_{1,2}-T_{1,3}\ge 1-T_{2,3}\Leftrightarrow T_{1,2}+T_{1,3}-T_{2,3}\le 1$. To prove the Lemma we generalize this map $\psi$.
\end{example}
\begin{proof}[Proof of Lemma~\ref{gtequiv}]
%It is evident from the definitions that for integral dominant weights $\la$ and $\mu$ we have both $P_\la+P_\mu=P_{\la+\mu}$ and $GT_\la+GT_\mu=GT_{\la+\mu}$. Therefore, in view of Proposition~\ref{gtminkowski}, it suffices to show that (1) $\Pi_{\om_k}$ is the set of integer points in $P_{\om_k}$ for all $1\le k\le n-1$ and (2) that there exists a linear unimodular automorphism $\psi$ of $Theta$ that preserves the lattice of integer points and identifies $P_{\om_k}$ with $GT_{\om_k}$ for all $1\le k\le n-1$. 
%
%The fact that $\Pi_{\om_k}\subset P_{\om_k}$ is immediate from the definitions. The fact that $P_{\om_k}$ has exactly $n\choose k$ integer points and therefore has no integer points outside of $\Pi_{\om_k}$ will follow from (2). Let us turn to defining the automorphism $\psi$. 
%
Consider the affine transformation $\psi$ of $\Theta$ given by 
\begin{equation}\label{psidef}
\psi(T)_{i,j}=\la_i-\sum_{l=i+n+1-j}^n T_{i,l}.
\end{equation}
It is evident that $\psi$ is unimodular and preserves the lattice of integer points, let us show that $\psi(P_\la)=GT_\la$. Indeed, if $j<n$, then the inequality in (i) is equivalent to 
\begin{equation}\label{psi1}
\psi(T)_{i,i+n-j}\ge \psi(T)_{i,i+n+1-j}
\end{equation} 
and the inequality in (ii) is equivalent 
\begin{equation}\label{psi2}
\psi(T)_{i,i+n+1-j}\ge \psi(T)_{i+1,i+n+1-j}.
\end{equation}
(\ref{psi1}) and (\ref{psi2}) combined over all $1\le i<j<n$ give (iv).
If $j=n$, then the inequality in (i) is equivalent to 
\begin{equation}\label{psi3}
\psi(T)_{i,i+1}\le\la_i
\end{equation}
and the inequality in (ii) is equivalent to 
\begin{equation}\label{psi4}
\psi(T)_{i,i+1}\ge\la_{i+1}.
\end{equation}
Combining~(\ref{psi3}) and~(\ref{psi4}) over all $1\le i\le n-1$ gives (iii). This proves the second part of the lemma.

Now, with the second part established, the first part follows from the definition of $\Pi_\la$, Proposition~\ref{gtminkowski} and the claim that $\Pi_{\om_k}$ is the set of integer points in $P_{\om_k}$ for all $1\le k\le n-1$. To verify this last claim note that $\Pi_{\om_k}\subset P_{\om_k}$ is immediate from the definitions and that, in view of the second part, $P_{\om_k}$ has exactly $n\choose k$ integer points and therefore has no integer points outside of $\Pi_{\om_k}$.
\end{proof}

We see that the map $\psi$ provides a bijection between $\Pi_\la$ and the set of GT patterns.
\begin{cor}\label{dimension}
$|\Pi_\la|=\dim L_\la$.
\end{cor}

\begin{remark} 
Lemma~\ref{gtequiv} is immediate from the results found in~\cite[Section 14.4]{MS}, in particular, a variation of our map $\psi$ is also constructed there. We give a self-contained proof for the sake of completeness. The formula~(\ref{psidef}) for $\psi$ is easily derived via additivity with respect to $\la$, a much more involved question is \emph{why} a polytope given by the fairly simple expression~(\ref{minksum}) would turn out to be equivalent to the GT polytope. One possible explanation can be found in~\cite[Section 5]{KM}. 
\end{remark}

Let us now define the aforementioned monomial bases. We make use of the following terminology. We call a monomial $M\in\mU(\fn_{-})$ {\it ordered} if the factors $f_{i,j}$ appearing in M are ordered by $i$ increasing from left to right. For every $T\in\Theta$ with nonnegative integer coordinates we consider the ordered monomial $M_T\in \mU(\fn_-)$ that contains $f_{i,j}$ in degree $T_{i,j}$. Note that $M_T$ is defined uniquely since any two elements of the form $f_{i,j_1}$ and $f_{i,j_2}$ commute.
\begin{theorem}\label{basis}
The set $\{M_T v_\la, T\in\Pi_\la\}$ is a basis in $L_\la$.
\end{theorem}
This theorem will be proved in Section~\ref{gtdegens} after we introduce the relevant degree functions on monomials in $\mU(\fn_-)$, but also cf.~\cite[Corollary 2.1.3]{ChL}.

\section{Gelfand--Tsetlin degenerations}\label{gtdegens}

In this section we will define a specific family of Gr\"obner degenerations of $F$ and list several properties of the corresponding objects introduced in Section~\ref{generalities} (as well as proving Theorem~\ref{basis}). First, consider a collection of integers $A=(a_{i,j}|1\le i<j\le n)$ such that 
\begin{enumerate}[label=(\alph*)]
\item $a_{i,i+1}+a_{i+1,i+2}\le a_{i,i+2}$ for any $1\le i\le n-2$ and
\item $a_{i,j}+a_{i+1,j+1}\le a_{i,j+1}+a_{i+1,j}$ for any $1\le i<j-1\le n-2$
\end{enumerate}
or, equivalently,
\begin{enumerate}[label=(\Alph*)]
\item $a_{i,j}+a_{j,k}\le a_{i,k}$ for any $1\le i<j<k\le n$ and
\item $a_{i,j}+a_{k,l}\le a_{i,l}+a_{k,j}$ for any $1\le i<k<j<l\le n$. 
\end{enumerate}
The proof that the inequalities in (A) and (B) can be deduced from those in (a) and (b) is straightforward and almost identical (up to reversing all inequalities) to the proof of Proposition~2.1 in~\cite{fafefom}.

We will view $A$ as an element of $\Theta^*$ equipped with the basis dual to the one chosen in $\Theta$. We define \[\sigma(A)=(\sigma(A)_{i_1,\ldots,i_k})\in \bR^{\{1\le i_1<\ldots<i_k\le n|1\le k\le n-1\}}\] with $\sigma(A)_{i_1,\ldots,i_k}=A(T(i_1,\ldots,i_k))$ (as a functional on $\Theta$). We will refer to Gr\"obner degenerations given by $\sigma(A)$ with $A$ satisfying (a) and (b) as Gelfand--Tsetlin (or GT) degenerations of $F$ and to the corresponding associated graded spaces $L_\la^{\sigma(A)}$ as GT degenerations of $L_\la$. From now on and through Section~\ref{singular} we fix $A$ and $S=(s_{i_1,\ldots,i_k})=\sigma(A)$.

\begin{example}\label{sl3ideal}
For $n=3$ one may set $A={{a_{1,2}\,a_{2,3}}\atop{a_{1,3}}}={{-1\,-1}\atop{-1}}$. The only inequality here is $a_{1,2}+a_{2,3}\le a_{1,3}$ and it, evidently, holds. All points $T(i_1,\ldots,i_k)$ are listed in Example~\ref{sl3polytopes} and one can compute $\grad^S X_1=\grad^S X_{1,2}=0$, $\grad^S X_2=\grad^S X_3=\grad^S X_{1,3}=-1$ and $\grad^S X_{2,3}=-2$. The ideal $I$ is generated by the element $X_1X_{2,3}-X_2X_{1,3}+X_3X_{1,2}$ and the initial part of this element with respect to $\grad^S$ is $X_1X_{2,3}-X_2X_{1,3}$ which is the sole generator of $\initial_{\grad^S}I$. 
\end{example}

For a monomial $M=f_{i_1,j_1}\ldots f_{i_N,j_N}\in\mU(\fn_{-})$ denote \[\deg^A M=a_{i_1,j_1}+\ldots+a_{i_N,j_N}.\] Note that when $n=3$ there are two ordered monomials in $\mU(\fn_-)$ which map $v_{\om_2}$ to a nonzero multiple of $e_{2,3}$, these are $M_{T(2,3)}$ and (in the notations of Example~\ref{sl3polytopes}) $M_{{0\,0}\atop{1}}$. We have \[\deg^A M_{T(2,3)}=a_{1,2}+a_{2,3}\le a_{1,3}=\deg^A M_{{0\,0}\atop{1}}.\] This inequality between the degrees of the monomials is a special case of the following key lemma which shows why we want $A$ to satisfy (A) and (B). 

%For $1\le k\le n-1$ let $\fn_k\subset\fn_{-}$ be the subalgebra spanned by $f_{i,j}$ with $i\le k$. 
\begin{lemma}\label{minmon}
For a tuple $1\le i_1<\ldots<i_k\le n$ and an ordered monomial $M$ such that $M v_{\om_k}\in\bC^*e_{i_1,\ldots,i_k}$ we have $\deg^A M\ge s_{i_1,\ldots,i_k}$.
\end{lemma}
\begin{proof}
First of all, note that $f_{i,j}$ maps $e_{i_1,\ldots,i_k}$ to $\pm e_{j_1,\ldots,j_k}$ where \[\{j_1,\ldots,j_k\}=\{i_1,\ldots,i_k\}\cup\{j\}\backslash\{i\}\] if $i\in\{i_1,\ldots,i_k\}$ and $j\notin\{i_1,\ldots,i_k\}$, otherwise $f_{i,j}$ maps $e_{i_1,\ldots,i_k}$ to 0. In particular, that means that for any monomial $M$ in the $f_{i,j}$ the vector $Mv_{\om_k}$ is either zero or of the form $\pm e_{j_1,\ldots,j_k}$.

Recall the ordered monomials $M_T$ defined in Section~\ref{bases}. Note that $M_{T(i_1,\ldots,i_k)} v_{\om_k}=\pm e_{i_1,\ldots,i_k}$ and that $\deg^A M_{T(i_1,\ldots,i_k)}=s_{i_1,\ldots,i_k}$. 

Now, consider an ordered monomial \[M=f_{l_1,m_1}\ldots f_{l_N,m_N}\] such that $Mv_{\om_k}=\pm e_{i_1,\ldots,i_k}$ with minimal $\deg^A M$ and out of these with the minimal possible sum $\sum_i (m_i-l_i)^2$. We prove the lemma by showing that $M=M_{T_{i_1,\ldots,i_k}}$.

Any product of the form $f_{i,j_1}f_{i,j_2}$ annihilates $L_{\om_k}$, therefore all $l_i$ are pairwise distinct. Suppose that for some $i$ we have $m_i\ge m_{i+1}$. The product $f_{l_i,m_i}f_{l_{i+1},m_i}$ annihilates $L_{\om_k}$, hence $m_i>m_{i+1}$. For any $e_{j_1,\ldots,j_k}$ we have \[f_{l_i,m_i}f_{l_{i+1},m_{i+1}}(e_{j_1,\ldots,j_k})=\pm f_{l_i,m_{i+1}}f_{l_{i+1},m_i}(e_{j_1,\ldots,j_k}).\] Therefore, by replacing $f_{l_i,m_i}f_{l_{i+1},m_{i+1}}$ in $M$ with $f_{l_i,m_{i+1}}f_{l_{i+1},m_i}$ we would obtain a monomial also mapping $v_{\om_k}$ to $\pm e_{i_1,\ldots,i_k}$ and of no greater $\deg^A$-degree due to (B). However, $$(m_i-l_{i+1})^2+(m_{i+1}-l_i)^2<(m_i-l_i)^2+(m_{i+1}-l_{i+1})^2$$ which would contradict our choice of $M$. We see that both sequences $(l_1,\ldots,l_N)$ and $(m_1,\ldots,m_N)$ are strictly increasing.

Now define a tuple $(j_1,\ldots,j_k)$ by setting $j_l=m_i$ if $l=l_i$ for some $i$ and $j_l=l$ otherwise. $Mv_{\om_k}=\pm e_{i_1,\ldots,i_k}$ implies that $(j_1,\ldots,j_k)$ is a permutation of $(i_1,\ldots,i_k)$ (note that all $l_i\le k$ since $M v_{\om_k}\neq 0$), we prove that $(j_1,\ldots,j_k)=(i_1,\ldots,i_k)$. Suppose the contrary, i.e.\ that $j_l>j_{l+1}$ for some $l$. The sequences $(l_1,\ldots,l_N)$ and $(m_1,\ldots,m_N)$ increasing implies that $j_{l+1}=l+1$ while $j_l=m_i$ for some $i$. In particular, $l+1\notin\{l_1,\ldots,l_N\}$ and replacing $f_{l,m_i}$ in $M$ with the product $f_{l,l+1}f_{l+1,m_i}$ we would obtain a monomial $M'$ with $M'v_{\om_k}=\pm Mv_{\om_k}$ and $\deg^AM'\le\deg^A M$ due to (A). However, $1+(m_i-l-1)^2<(m_i-l)^2$ which, again, contradicts our choice of $M$.
\end{proof}

The above proof has the following two implications.
\begin{proposition}\label{minmonstrict}
Suppose that all the inequalities in (A) and (B) are strict. If $T\in\bZ_{\ge0}^{\{1\le i<j\le n\}}\subset\Theta$ is such that $M_Tv_{\om_k}=\pm e_{i_1,\ldots,i_k}$, then either $T=T(i_1,\ldots,i_k)$ or $\deg^A M_T>s_{i_1,\ldots,i_k}$.
\end{proposition}
For $T\in\Theta$ denote $\sq(T)=\sum_{i,j} T_{i,j}(j-i)^2$. 
\begin{proposition}\label{sumsquares}
If $T\in\bZ_{\ge0}^{\{1\le i<j\le n\}}\subset\Theta$ is such that $M_Tv_{\om_k}=\pm e_{i_1,\ldots,i_k}$ and $\deg^A M_T=s_{i_1,\ldots,i_k}$, then either $T=T(i_1,\ldots,i_k)$ or $\sq(T)>\sq(T(i_1,\ldots,i_k))$.
\end{proposition}

We are now ready to prove Theorem~\ref{basis}.
\begin{proof}[Proof of Theorem~\ref{basis}]
In view of Corollary~\ref{dimension} it suffices to show that the set $\{M_T v_\la,T\in\Pi_\la\}$ is linearly independent. We make use of $L_\la$ being embedded into $U_\la$ as the subrepresentation generated by $u_\la$. 

For $T\in\Pi_\la$ let $U_T\subset U_\la$ be the subspace spanned by products of the form \[v^1_1\otimes\ldots\otimes v^1_{a_1}\otimes\ldots\otimes v^{n-1}_1\otimes\ldots\otimes v^{n-1}_{a_{n-1}}\] with $v^i_j=e_{l^{i,j}_1,\ldots,l^{i,j}_i}$ for which the total of all $T(l^{i,j}_1,\ldots,l^{i,j}_i)$ is equal to $T$. Then $U_\la$ is the direct sum of $U_T$ with $T$ ranging over $\Pi_\la$ and we see that every $U_T$ is $\grad^S$-homogeneous with $U_T\subset (U_\la)_{A(T)}$.

Now, choose $T\in\Pi_\la$ and decompose $M_T u_\la$ into a sum of tensor products. Every summand is obtained by partitioning the set of factors in $M_T$ into $a_1+\ldots+a_{n-1}$ subsets, one for every tensor factor, applying the ordered product of each subset to the corresponding $v_{\om_k}$ and then taking the tensor product of the results. %Here note that if a monomial being applied to some $v_{\om_k}$ contains some $f_{i,j}$ with $i>k$, then the result of this application and, subsequently, the whole summand is zero. 
We now see, by applying Lemma~\ref{minmon} and Proposition~\ref{sumsquares}, that for every summand in this decomposition one of the following holds. It either a) lies in $U_T$, b) lies in some $U_{T'}$ with $A(T')=A(T)=\deg^A M_T$ and $\sq(T')<\sq(T)$ or c) lies in some $U(T')$ with $A(T')<A(T)$. Moreover, in view of the Minkowski sum property, a) holds for at least one of the summands and, therefore, the projection of $M_T u_\la$ onto $U_T$ is nonzero.

Choose a total order $\prec$ on $\Pi_\la$ such that $T_1\prec T_2$ whenever $A(T_1)<A(T_2)$ or $A(T_1)=A(T_2)$ and $\sq(T_1)<\sq(T_2)$. For a nontrivial linear combination $\Omega$ of $M_Tu_\la$ with $T\in\Pi_\la$ let $T_0$ be the $\prec$-maximal element for which $M_{T_0}u_\la$ appears in $\Omega$ with a nonzero coefficient. Since the projection onto $U_{T_0}$ is nonzero for $M_{T_0}u_\la$ but zero for all other summands, it is also nonzero for $\Omega$.
\end{proof}
For $T\in\bZ_{\ge0}^{\{1\le i<j\le n\}}$ say that the monomial $M_T$ is {\it $L_\la$-optimal} if $M_T v_\la$ lies in $(L_\la)_{A(T)}$ but not in $(L_\la)_{A(T)-1}$. The above proof implies the following fact which we will make use of later.
\begin{proposition}\label{basiscor}
For every $T\in\Pi_\la$ the monomial $M_T$ is $L_\la$-optimal.
\end{proposition}

We now explicitly describe a filtration on $\mU(\fn_-)$ which induces the filtrations on $L_\la$ given by a GT degeneration. The increasing $\bZ$-filtration (but not a filtered algebra structure!) on $\mU(\fn_{-})$ is defined by component $\mU(\fn_{-})_m$ being spanned by ordered monomials $M$ with $\deg^AM\le m$. Recall the filtration $((L_\la)_m, m\in\bZ)$ defined in Section~\ref{generalities} (with respect to the chosen $S=\sigma(A)$).
\begin{theorem}\label{Lfiltration}
$(\mU(\fn_-))_mv_\la=(L_\la)_m$ for every $m\in\bZ$.
\end{theorem}
\begin{proof}
Similarly to the proof of Theorem~\ref{basis}, for a $T\in\bZ_{\ge0}^{\{1\le i<j\le n\}}$ the vector $M_T u_\la$ is a sum of tensor products each lying in some $U_m$ with $m\le \deg^A M_T$. This gives the inclusion $(\mU_\la)_mv_\la\subset(L_\la)_m$.

For the reverse inclusion consider a vector $v\in(L_\la)_m$ and express \[v=\sum_{T\in\Pi_\la} c_TM_Tv_\la.\] Among the $T$ with $c_T\neq 0$ choose a $T_0$ which has the maximal $A(T)=\deg^A M_T$ and among those with maximal $A(T)$ has the minimal $\sq(T)$. From the proof of Theorem~\ref{basis} we see that the projection of $v$ onto the direct summand $U_{T_0}$ is nonzero. This implies that $A(T_0)=\deg^A M_{T_0}\le m$ and, consequently, $v\in(\mU(\fn_-))_m$.
\end{proof}

We proceed to give two characterizations of the initial ideal $\initial_{\grad^S} I$ which mimic and follow from well known characterizations of the ideal of Pl\"ucker relations $I$.

Let us consider the polynomial ring $Q=\bC[\{z_{i,j},1\le i\le j\le n\}]$. On this ring we have a grading $\grad^A$ given by $\grad^A z_{i,j}=a_{i,j}$ if $i<j$ and $\grad^A z_{i,i}=0$. Let $\zeta$ be the $n\times n$ matrix with $\zeta_{i,j}=z_{i,j}$ if $i\le j$ and $\zeta_{i,j}=0$ otherwise. Let $D_{i_1,\ldots,i_k}\in Q$ be the determinant of the submatrix of $\zeta$ spanned by the first $k$ rows and columns $i_1,\ldots,i_k$. 

First, a fact concerning non-degenerate flag varieties.
\begin{proposition}\label{kerneldetclassic}
$I$ is the kernel of the map $\delta$ from $R$ to $Q$ taking $X_{i_1,\ldots,i_k}$ to $D_{i_1,\ldots,i_k}$.
\end{proposition}
\begin{proof}
This is a variation of the following classical fact which can be interpreted as a definition of the Pl\"ucker embedding. If we introduce $n\choose2$ more variables $z_{i,j}$ for $1\le j<i\le n$, consider the matrix $\zeta'$ with $\zeta_{i,j}=z_{i,j}$ and let $D'_{i_1,\ldots,i_k}$ be the same minor but in $\zeta'$, then $I$ is the kernel of the map $\delta^0$ from $R$ to $\bC[z_{i,j},1\le i,j\le n]$ taking $X_{i_1,\ldots,i_k}$ to $D'_{i_1,\ldots,i_k}$.

The map $\delta$ is the composition of $\delta^0$ and the map from $\bC[z_{i,j},1\le i,j\le n]$ to Q taking $z_{i,j}$ to $0$ if $j<i$ and to $z_{i,j}$ if $i\le j$. Therefore, the kernel of $\delta$ contains $I$.

Now, $F$ can be viewed as $GL_n/B'$, where $B'$ is the set of lower triangular matrices. If we consider a matrix $z\in GL_n$ and specialize the variables $z_{i,j}$ to the elements of this matrix, then the image of $z$ under the projection $GL_n\to F\subset \bP$ will have homogeneous coordinates $(D'_{i_1,\ldots,i_k},1\le i_1<\ldots<i_k\le n)$ which coincides with $(D_{i_1,\ldots,i_k},1\le i_1<\ldots<i_k\le n)$ if $z$ is upper triangular. Therefore, any polynomial $p\in R$ with $\delta(p)=0$ vanishes on the subset of $F$ that is the image of the set of upper triangular matrices in $GL_n$. However, the latter image is Zariski dense in F and we obtain $p\in I$.
\end{proof}

Now we give an analogous fact for GT degenerations.
\begin{theorem}\label{kerneldet}
$\initial_{\grad^S} I$ is the kernel of the map $\delta^S$ from $R$ to $Q$ sending $X_{i_1,\ldots,i_k}$ to $\initial_{\grad^A} D_{i_1,\ldots,i_k}$.
\end{theorem}
\begin{proof}
For a monomial $p\in Q$ let $T(p)\in\Theta$ be the point with coordinate $T(p)_{i,j}$ equal to the degree of $z_{i,j}$ in $p$. Observe that for every monomial $p$ appearing in the polynomial $D_{i_1,\ldots,i_k}$ we have $M_{T(p)} v_{\om_k}=\pm e_{i_1,\ldots,i_k}$. Exactly one of those monomials $q$ has $T(q)=T(i_1,\ldots,i_k)$, hence we see that $\grad^A(\initial_{\grad^A} D_{i_1,\ldots,i_k})=s_{i_1,\ldots,i_k}$ and that $\delta^S(X_{i_1,\ldots,i_k})=\initial_{\grad^A} D_{i_1,\ldots,i_k}$ is a sum of $q$ and other monomials $p$ with $\sq(T(p))>\sq(T(i_1,\ldots,i_k))$.

The fact that $\grad^A(\initial_{\grad^A} D_{i_1,\ldots,i_k})=\grad^S(X_{i_1,\ldots,i_k})$ implies (via Proposition~\ref{kerneldetclassic}) that the kernel of $\delta^S$ contains $\initial_{\grad^S} I$. To prove the reverse inclusion we show that the graded components of $\delta^S(R)$ have dimensions no less than those of $\mathcal P$. Namely, for a integral dominant weight $\la$ let $Q(\la)$ be spanned by those monomials that for every $1\le i\le n$ contain all variables of the form $z_{i,j}$ in total degree $\la_i$ (the $\la_i$ were defined in (iii) in Section~\ref{bases}). One sees that $\delta^S$ maps $R_\la$ into $Q(\la)$. The somewhat inconsistent notation is caused by the fact that a slightly different grading on $Q$ by weights will be considered below. 

Choose a $\la=(a_1,\ldots,a_{n-1})$ and and some $T\in\Pi_\la$. Combining Proposition~\ref{gtminkowski} and Lemma~\ref{gtequiv} we can decompose \[T=T^1_1+\ldots+T^1_{a_1}+\ldots+T^{n-1}_1+\ldots+T^{n-1}_{a_{n-1}}\] where $T^i_j\in\Pi_{\om_i}$. For $T^i_j=T(i_1,\ldots,i_k)$ denote $X^i_j=X_{i_1,\ldots,i_k}$ and consider the monomial $Y_T=\prod_{i,j} X^i_j\in R_\la$. From the first paragraph of the proof we see that $\delta^S(Y_T)$ is the sum of a monomial $q$ with $T(q)=T$ and other monomials $p$ with $\sq(T(p))<\sq(T)$. Consequently, the expressions $\delta^S(Y_T)$ with $T$ ranging over $\Pi_\la$ are linearly independent and the proposition follows.
\end{proof}

In Example~\ref{sl3polytopes} we see that the only point in $\Pi_\la$ that can be decomposed into a sum of points in $\Pi_{\om_1}$ and $\Pi_{\om_2}$ in two different ways is ${{1\,1}\atop0}=T(1)+T(2,3)=T(2)+T(1,3)$. One can deduce that the toric variety associated with the polytope $P_\la$ can be embedded into $\bP$ as the set of zeros of the ideal $\langle X_1X_{2,3}-X_2X_{1,3}\rangle$. However, this ideal coincides with $\initial_{\grad^S}I$ obtained in Example~\ref{sl3ideal}. We generalize this to a fact that is one of our main reasons for considering these degenerations and terming them ``Gelfand--Tsetlin degenerations''.
\begin{theorem}\label{toric}
If all inequalities in (a) and (b) (equivalently, all inequalities in (A) and (B)) are strict and $\la$ is regular, then the GT degeneration $F^S$ is the toric variety associated with the polytope $P_\la$. This is isomorphic to the toric variety associated with the Gelfand--Tsetlin polytope $GT_\la$.
\end{theorem}
\begin{proof}
As pointed out in the proof of Theorem~\ref{kerneldet}, for every monomial $p$ appearing in the polynomial $D_{i_1,\ldots,i_k}$ we have $M_{T(p)} v_{\om_k}=\pm e_{i_1,\ldots,i_k}$. However, in view of Proposition~\ref{minmonstrict}, if all inequalities in (a) and (b) are strict, then $M_{T(i_1,\ldots,i_k)}$ is the only ordered $L_{\om_k}$-optimal monomial mapping $v_{\om_k}$ to $\pm e_{i_1,\ldots,i_k}$. We deduce that $\initial_{\grad^A} D_{i_1,\ldots,i_k}=\prod_{l=1}^k z_{l,i_l}$.

The fact that the subring in $Q$ generated by the monomials $\prod_{l=1}^k z_{l,i_l}$ is the coordinate ring of the toric variety in question is essentially proved in~\cite[Chapter 14]{MS}. However, we can observe that this subring is the semigroup ring of the semigroup in $\fh^*\oplus\Theta$ generated by points of the form $(\om_k,T(i_1,\ldots,i_k))$. This semigroup ring is the homogeneous coordinate ring of the toric variety associated with $P_\la$. The second claim in the proposition follows from the unimodular equivalence proved in Lemma~\ref{gtequiv}.
\end{proof}
The toric ideal $J$ obtained as the kernel of the map taking $X_{i_1,\ldots,i_k}$ to $\prod_{l=1}^k z_{l,i_l}$ is precisely the ideal considered in the works~\cite{GL} and~\cite{KM}. We will come back to this ideal in Sections~\ref{tropical} and~\ref{addendum}.

We move on to the second characterization. Choose a complex vector $c=(c_{i,j})\in\bC^{\{1\le i<j\le n\}}$ and consider the $G$-action $v_k(c)=\prod_{i,j} \exp(c_{i,j}f_{i,j})v_{\om_k}\in L_{\om_k}$, where factors in the product are ordered by $i$ increasing from left to right. This defines $v_k(c)$ uniquely in view of the commutation relations. The coordinate of $v_k(c)$ corresponding to basis vector $e_{i_1,\ldots,i_k}$ is equal to $C_{i_1,\ldots,i_k}(c)$ for some polynomial $C_{i_1,\ldots,i_k}\in\bC[z_{i,j},1\le i<j\le n]$. In the non-degenerate case the following holds. 
\begin{proposition}\label{expprodclassic}
$I$ is the kernel of the map $\varepsilon$ from $R$ to $Q$ sending $X_{i_1,\ldots,i_k}$ to $z_{k,k}C_{i_1,\ldots,i_k}$.
\end{proposition}
\begin{proof}
For an integral dominant weight $\la$ let $\bm v_\la\in\bP(L_\la)$ be the point corresponding to $\bC v_\la$. Let $N\subset G$ be the unipotent subgroup with tangent algebra $\fn_-$. $N$ acts on $\bP$ and the closure of the orbit $N\bm v$ is $F$ where $\bm v=\bm v_{\om_1}\times\ldots\times\bm v_{\om_{n-1}}$. Now, the Pl\"ucker coordinates of the point $\prod_{i,j}\exp(c_{i,j}f_{i,j})(\bm v)$ are precisely $C_{i_1,\ldots,i_k}$(c). In view of the additional factor $z_{k,k}$, the kernel of $\varepsilon$ is a $\wt$-homogeneous ideal that contains $I$.

We are left to show that the set of points of the form $\prod_{i,j}\exp(c_{i,j}f_{i,j})(\bm v)$ is open in $F$ or, sufficiently, that the set of products of the form $\prod_{i,j}\exp(c_{i,j}f_{i,j})$ is open in $N$. In fact, induction on $n$ easily shows that the set of such products is all of $N$. For the induction step one writes $N=N_{n-1}\exp(\fn_1)$ where $N_{n-1}$ is the exponential of the subalgebra spanned by $f_{i,j}$ with $i>1$. 
\end{proof}

Now, our analog for GT degenerations.
\begin{theorem}\label{kernelexp}
$\initial_{\grad^S} I$ is the kernel of the map $\varepsilon^S$ from $R$ to $Q$ sending $X_{i_1,\ldots,i_k}$ to $\initial_{\grad^A}(z_{k,k}C_{i_1,\ldots,i_k})$. 
\end{theorem}
\begin{proof}
Consider a grading on $Q$ with $Q_\la$ being spanned by those monomials that for every $1\le i\le n-1$ contain the variable $z_{i,i}$ in degree $a_i$. Once again, for every monomial $p$ appearing in the polynomial $z_{k,k}C_{i_1,\ldots,i_k}$ we have $M_{T(p)} v_{\om_k}=\pm e_{i_1,\ldots,i_k}$ and exactly one of these monomials $q$ has $T(q)=T(i_1,\ldots,i_k)$. The rest of the proof repeats that of Theorem~\ref{kerneldet} verbatim modulo the appropriate substitutions.
\end{proof}

\section{The degenerate action}\label{degenact}

In this section we define an associative algebra that acts on the GT degenerate representation spaces $L_\la^S$ and give an explicit description of the embedding of $F^S$ into $L_\la^S$ in terms of this action. 

Let us consider the associative algebra $\Phi_n$ generated by elements $\{\varphi_{i,j}|1\le i<j\le n\}$ with relations $\varphi_{i_1,j_1}\varphi_{i_2,j_2}=0$ whenever $i_1>i_2$ and $\varphi_{i_,j_1}\varphi_{i_,j_2}=\varphi_{i,j_2}\varphi_{i,j_1}$ for all $1\le i<j_1<j_2\le n$. For $T\in\bZ_{\ge 0}^{\{1\le i<j\le n\}}$ let $\varphi^T\in\Phi_n$ be the product $\prod_{i,j}\varphi_{i,j}^{T_{i,j}}$ with the factors ordered by $i$ increasing from left to right (which defines $\varphi^T$ uniquely). The elements $\varphi^T$ form a basis in $\Phi_n$.

We define an action of $\Phi_n$ on the vector space $L_\la$. To do so for $1\le k\le n-1$ consider the Lie algebra $\fn_-(k)\subset\fn_-$ spanned by $f_{i,j}$ with $i\ge k$, we see that $\fn_-(1)=\fn_-$ and that $\fn_-(k)$ is a nilpotent subalgebra in $\msl_{n-k+1}$. Denote $L_\la(k)=\mU(\fn_-(k))v_\la\subset L_\la$. Note that the root vectors $-\alpha_{i,j}$ with $i\ge k$ generate a simple cone $\mathfrak c(k)\subset\fh^*$ of dimension $n-k$ with edges generated by $\alpha_i$ with $i\ge k$. One sees that $L_\la(k)$ is precisely the sum of all weight subspaces in $L_\la$ of weights $\mu$ for which $\mu-\la\in\mathfrak c(k)$.

Our action is defined as follows. For each $\varphi_{i,j}$ and a weight vector $v\in L_\lambda$ we have $\varphi_{i,j}v=f_{i,j}v$ if $v\in L_\la(i)$ and $\varphi_{i,j}v=0$ otherwise. 
\begin{proposition}\label{phionL}
This is a well-defined $\Phi_n$-module structure on $L_\la$. For every $T\in\bZ_{\ge 0}^{\{1\le i<j\le n\}}$ we have $\varphi^T v_\la=M_T v_\la$.
\end{proposition}
\begin{proof}
We verify that the considered endomorphisms of $L_\la$ satisfy the defining relations for $\Phi_n$. The image of $f_{i,j}$ intersects $L_\la(i-1)$ trivially, consequently, so does the image of $\varphi_{i,j}$. This implies that the first set of relations is satisfied. The actions of $\varphi_{i,j_1}$ and $\varphi_{i,j_2}$ commute since they both annihilate anything outside of $L_\la(i)$ and therefore may be viewed as commuting endomorphisms of $L_\la(i)$.

The second claim is easily obtained by induction on $\sum_{i,j} T_{i,j}$ via the fact that $\varphi_{i,j}$ preserves $L_\la(i)$.
\end{proof}

Now we introduce a $\bZ$-grading on the algebra $\Phi_n$ by setting $\deg^A\varphi_{i,j}=a_{i,j}$, denote by $\Phi_{n,m}$ the homogeneous components of this grading. Subsequently we obtain an increasing $\bZ$-filtration on $\Phi_n$ with components $\Phi_{n,\le m}=\bigoplus_{l\le m}\Phi_{n,l}$. On one hand, this is a filtered algebra and the associated graded algebra is again $\Phi_n$ with the same grading $\deg^A$. On the other, this filtration induces a filtration on $L_\la$ via $(L_\la)_{\le m}=\Phi_{n,\le m}v_\la$.
\begin{proposition}\label{filtrationsame}
$(L_\la)_{\le m}=(L_\la)_m$, i.e.\ the newly introduced filtration coincides with the one considered previously.
\end{proposition}
\begin{proof}
This is immediate from the second part of Proposition~\ref{phionL} and Theorem~\ref{Lfiltration}.
\end{proof}
Let us turn the associated graded space $L_\la^S$ into a $\Phi_n$-module by degenerating the action on $L_\la$. We have the surjections $(L_\la)_m\to (L_\la^S)_m$ with kernels $(L_\la)_{m-1}$ and the maps $\varphi_{i,j}:(L_\la)_m\to(L_\la)_{m+a_{i,j}}$. This induces maps $\varphi_{i,j}:(L_\la^S)_m\to(L_\la^S)_{m+a_{i,j}}$ which are summed over $m$ to provide maps $\varphi_{i,j}:L_\la^S\to L_\la^S$. 

Let $v_\la^S$ be the image of $v_\la\in (L_\la)_0$ in $(L_\la^S)_0$.
\begin{proposition}\label{phionLS}
This is a well-defined $\Phi_n$-module structure on $L_\la^S$. The action on $v_\la^S$ is described as follows. For $T\in\bZ_{\ge 0}^{\{1\le i<j\le n\}}$ the vector $\varphi^T v_\la^S$ is the projection of $M_T v_\la\in(L_\la)_{A(T)}$ to $(L_\la^S)_{A(T)}$ (thus $\varphi^T v_\la^S=0$ if $M_T$ is not $L_\la$-optimal).
\end{proposition}
\begin{proof}
To compute the image of a vector $v\in(L_\la^S)_m$ under the action of $\varphi_{i,j}$ one may choose a preimage $v'\in(L_\la)_m$ of $v$ and take the image of $\varphi_{i,j}v'\in(L_\la)_{m+a_{i,j}}$ in $(L_\la^S)_{m+a_{i,j}}$. The first claim now follows from the first part of Proposition~\ref{phionL}.

$\varphi^T v_\la^S$ is the projection of $\varphi^T v_\la\in (L_\la)_{A(T)}$ to $(L_\la^S)_{A(T)}$ and $\varphi^T v_\la=M_T v_\la$, therefrom we obtain the second claim.
\end{proof}

In particular, the above proposition combined with Proposition~\ref{basiscor} have the following consequence.
\begin{cor}\label{degenbasis}
For an integral dominant weight $\la$ the set of vectors $\{\varphi^Tv_\la^S,T\in\Pi_\la\}$ is a basis in $L_\la^S$.
\end{cor}

Our next goal is, given a $\la=(a_1,\ldots,a_{n-1})$, to define a $\Phi_n$-module structure on \[U_\la^S=(L_{\omega_1}^S)^{\otimes a_1}\otimes\ldots\otimes (L_{\omega_{n-1}}^S)^{\otimes a_{n-1}}.\] In fact, we define a monoidal structure on the category $\mathcal C$ of finite-dimensional $\Phi_n$-modules $L$ that are also equipped with an $\fh$-action with the following properties. First, $L$ is a direct sum of its $\fh$-weight spaces. Second, for a $\fh$-weight vector $v\in L$ and $h\in\fh$ we one has $h(\varphi_{i,j}(v))=\varphi_{i,j}(h(v))-h(\alpha_{i,j})\varphi_{i,j}(v)$, i.e.\ $\varphi_{i,j}$ decreases the weight by $\alpha_{i,j}$. Finally, among all weights with nonzero multiplicities in $L$ there exists a single highest weight $\la$ (i.e. $\la$ is obtained from any other by adding a sum of positive roots) such that $\varphi_{i,j}(v)=0$ for any $\varphi_{i,j}$ and any weight vector $v\in L$ of weight $\mu$ such that $\mu-\la\notin\mathfrak c(i)$. We refer to this $\la$ as the highest weight of $L$. It is easily seen that each $L_\la^S$ inherits a weight decomposition from $L_\la$, lies in $\mathcal C$ and has highest weight $\la$.

%First we observe that each $L_{w_k}^S$ is equiped with an action of the Cartan subalgebra $\mathfrak h$. To do so we note that a vector $e_{i_1,\ldots,i_k}$ lies in $(L_{\om_k})_{s_{i_1,\ldots,i_k}}$ but not in $(L_{\om_k})_{s_{i_1,\ldots,i_k}-1}$ (due to Lemma~\ref{minmon}), let $e_{i_1,\ldots,i_k}^S$ be the image of this vector in $(L_{w_k}^S)_{s_{i_1,\ldots,i_k}}$. It now suffices to say that $e_{i_1,\ldots,i_k}^S$ has the same weight as $e_{i_1,\ldots,i_k}$, since the vectors $e_{i_1,\ldots,i_k}^S$ form a basis in $L_{w_k}^S$. This structure induces an action of $\mathfrak h$ on the tensor product $U_\la^S$ which decomposes into the direct sum of its weight subspaces.

For modules $L_1$, $L_2$ in $\mathcal C$ with highest weights $\la_1$, $\la_2$ denote $U=L_1\otimes L_2$. Note that a weight decomposition is induced on $U$ and for $1\le k\le n-1$ let $U(k)$ be the sum of weight subspaces in $U$ of weights $\mu$ such that $\mu-(\la_1+\la_2)\in\mathfrak c(k)$. For weight vectors $v_1\in L_1$, $v_2\in L_2$ and each $\varphi_{i,j}$ %and a product of weight vectors \[v=v^1_1\otimes\ldots\otimes v^1_{a_1}\otimes\ldots\otimes v^{n-1}_1\otimes\ldots\otimes v^{n-1}_{a_{n-1}}\in U_\la^S\] 
we set $\varphi_{i,j}(v_1\otimes v_2)=0$ if $v_1\otimes v_2\notin U(i)$, otherwise we set 
\begin{equation}\label{phionUdef}
\varphi_{i,j}(v_1\otimes v_2)=\varphi_{i,j}(v_1)\otimes v_2+v_1\otimes\varphi_{i,j}(v_2).%\sum_{k, l} v^1_1\otimes \ldots\otimes\varphi_{i,j}(v^k_l)\otimes\ldots\otimes v^{n-1}_{a_{n-1}},
\end{equation}
%i.e the sum of the expressions obtained from $v$ by applying $\varphi_{i,j}$ to each of the tensor factors.
\begin{proposition}
This is a well-defined $\Phi_n$-module structure on $U$. 
\end{proposition} 
\begin{proof}
We see that the action of $\varphi_{i,j}$ subtracts $\alpha_{i,j}$ from the weight of a weight vector in $U$ and that the image of this action lies in $U(i)$ but intersects $U(i+1)$ trivially. We deduce $\varphi_{i_1,j_1}\varphi_{i_2,j_2}U=0$ whenever $i_1>i_2$. The commutation of the actions of $\varphi_{i,j_1}$ and $\varphi_{i,j_2}$ on $U$ follows from the definition~(\ref{phionUdef}) and the fact that they commute on $L_1$ and $L_2$. 
\end{proof}
It is now obvious that $U$ lies in $\mathcal C$ with highest weight $\la_1+\la_2$ and that the defined tensor product in $\mathcal C$ is associative and symmetric. In particular, this lets us view $U_\la^S$ as a $\Phi_n$-module.

%\begin{remark}
%One could consider the category of finite-dimensional $\Phi_n$-modules $L$ that are also equipped with an $\fh$-action in such a way that for a $\fh$-weight vector $v\in L$ and $h\in\fh$ we one has $h(\varphi_{i,j}(v))=\varphi_{i,j}(h(v))-h(\alpha_{i,j})\varphi_{i,j}(v)$, i.e.\ $\varphi_{i,j}$ decreases the weight by $\alpha_{i,j}$. All of the $\Phi_n$-modules we consider have a natural weight structure and lie in this category. The above tensor product construction generalizes straightforwardly to provide a monoidal structure on this category, we, however, will not need this kind of generality.
%\end{remark}

Next, note that a vector $e_{i_1,\ldots,i_k}$ lies in $(L_{\om_k})_{s_{i_1,\ldots,i_k}}$ but not in $(L_{\om_k})_{s_{i_1,\ldots,i_k}-1}$ (due to Lemma~\ref{minmon}), let $e_{i_1,\ldots,i_k}^S$ be the image of this vector in $(L_{w_k}^S)_{s_{i_1,\ldots,i_k}}$. We have linear isomorphisms between $L_{\om_k}$ and $L_{\om_k}^S$ sending $e_{i_1,\ldots,i_k}$ to $e_{i_1,\ldots,i_k}^S$ which induce a linear isomorphism between $U_\la$ and $U_\la^S$. Now recall the embedding $L_\la^S\subset U_\la$ from Section~\ref{generalities}. Denote the composition of the latter embedding and former isomorphism $\iota:L_\la^S\hookrightarrow U_\la^S$. Note that \[\iota(v_\la^S)=u_\la^S=(v_{\omega_1}^S)^{\otimes a_1}\otimes\ldots\otimes (v_{\omega_{n-1}}^S)^{\otimes a_{n-1}}.\]
\begin{lemma}
The embedding $\iota$ is a homomorphism of $\Phi_n$-modules.
\end{lemma}
\begin{proof}
Since $L_\la^S$ is generated by $v_\la^S$ as a $\Phi_n$ module (due to Corollary~\ref{degenbasis}), it suffices to show that for any $\varphi^T$ we have $\iota(\varphi^T v_\la^S)=\varphi^T u_\la^S$.

The vector $\varphi^T u_\la^S$ can be written explicitly as 
\begin{equation}\label{phiTulaS}
\sum_{\sum T^k_l=T} (\varphi^{T_1^1} v_{\om_1}^S)\otimes \ldots\otimes(\varphi^{T_l^k} v_{\om_k}^S)\otimes\ldots\otimes (\varphi^{T^{n-1}_{a_{n-1}}} v_{\om_{n-1}}^S).
\end{equation}
Here we sum over all decompositions of $T$ into a sum of $T^k_l\in\bZ^{\{1\le i<j\le n\}}$ with $1\le k\le n-1$ and $1\le l\le a_k$. Note that, in view Proposition~\ref{phionLS}, only those summands are nonzero in which each of the monomials $M_{T^k_l}$ is $L_{\om_k}$-optimal.

Now consider $\iota(\varphi^T v_\la^S)$. The image of $\varphi^T v_\la^S\in L_\la^S$ under the embedding into $U_\la$ is seen to coincide with the projection of $M_T u_\la\in(L_\la)_{A(T)}\subset(U_\la)_{\le A(T)}$ to $(L_\la^S)_{A(T)}\subset(U_\la)_{A(T)}$ due to Proposition~\ref{phionLS}. Now, \[M_T u_\la=\sum_{\sum T^k_l=T} (M_{T_1^1} v_{\om_1})\otimes \ldots\otimes(M_{T_l^k} v_{\om_k})\otimes\ldots\otimes (M_{T^{n-1}_{a_{n-1}}} v_{\om_{n-1}})\] with $\{T^k_l\}$ ranging over the same set of partitions as in~(\ref{phiTulaS}). Observe that unless each $M_{T^k_l}$ is $L_{\om_k}$-optimal in a summand, this summand lies in $(U_\la)_{\le A(T)-1}$. Therefore, when taking the projection onto $(L_\la)_{A(T)}\subset(U_\la)_{\le A(T)}$ only those summands in which each $M_{T^k_l}$ is $L_{\om_k}$-optimal remain. 

Finally observe that if for a $L_{\om_k}$-optimal $M_{T^k_l}$ we have $M_{T_l^k} v_{\om_k}=\pm e_{i_1,\ldots,i_k}$, then, due to Proposition~\ref{phionLS}, $\varphi^{T_l^k} v_{\om_k}^S=\pm e_{i_1,\ldots,i_k}^S$. Thus our bijection from $L_{\om_k}$ to $L_{\om_k}^S$ maps $M_{T_l^k} v_{\om_k}$ to $\varphi^{T_l^k} v_{\om_k}^S$ and the assertion follows.
\end{proof}

\begin{cor}\label{degcartcomp}
The $\Phi_n$-submodule in $U_\la^S$ generated by $u_\la^S$ is isomorphic to $L_\la^S$ (as a $\Phi_n$-module).
\end{cor}

Next, consider a complex vector $c=(c_{i,j},1\le i<j\le n)$. It is evident that each element $c_{i,j}\varphi_{i,j}$ acts nilpotently in the $\Phi_n$-modules $L_\la^S$ and $U_\la^S$ which allows us to consider the exponential of its action. We denote this exponential simply $\exp(c_{i,j}\varphi_{i,j})$. Furthermore, in each of these $\Phi_n$-modules we introduce the operator \[\exp(c)=\prod_{i,j}\exp(c_{i,j}\varphi_{i,j})\] where the factors are are ordered by $i$ increasing from left to right (which defines $\exp(c)$ uniquely). We may now straightforwardly transfer the actions of $\exp(c_{i,j}\varphi_{i,j})$ and $\exp(c)$ to the projectivizations of said $\Phi_n$-modules.

Let and $\bm v_\la^S$ be the point in $\bP(L_\la^S)$ corresponding to $\bC v_\la^S$. The following theorem is what we view as the main result of this paper.
\begin{theorem}\label{main}
For an integral dominant regular weight $\la$ let $E_\la$ be the image of $\bC^{\{1\le i<j\le n\}}$ in $\bP(L_\la^S)$ under the map taking $c$ to $\exp(c)\bm v_\la^S$. The Zariski closure of $E_\la$ is the degenerate flag variety $F^S$. This embedding of $F^S$ into $\bP(L_\la^S)$ coincides with the one given by Proposition~\ref{FSinPLS}
\end{theorem}

The following fact is crucial to our proof of the theorem.
\begin{lemma}\label{exptensor}
For any $c\in\bC^{\{1\le i<j\le n\}}$ and $\la=(a_1,\ldots,a_{n-1})$ we have \[\exp(c)u_\la^S=(\exp(c)v_{\omega_1}^S)^{\otimes a_1}\otimes\ldots\otimes (\exp(c)v_{\omega_{n-1}}^S)^{\otimes a_{n-1}}\in U_\la^S.\]
\end{lemma}
\begin{proof}
Since $\varphi_{i,j}$ acts on all $L_{\om_k}^S(i)$ and on $U_\la^S(i)$, so does the one-dimensional Lie algebra $\bC\varphi_{i,j}$. By definition, the action of this Lie algebra on $U_\la^S(i)$ is the tensor product of its actions on the $L_{\om_k}^S(i)$. Therefore the Lie group $\exp(\bC\varphi_{i,j})=\mathbb G_a$ (i.e.\ $\bC$ under addition) also acts on these spaces and its action on $U_\la^S(i)$ is the tensor product of its actions on the $L_{\om_k}^S(i)$. This means that for any \[v=v^1_1\otimes\ldots\otimes v^{n-1}_{a_{n-1}}\in U_\la^S(i)\] we have \[\exp(c_{i,j}\varphi_{i,j})v=\exp(c_{i,j}\varphi_{i,j})v^1_1\otimes\ldots\otimes \exp(c_{i,j}\varphi_{i,j})v^{n-1}_{a_{n-1}}.\] 

Now for $1\le k\le n$ denote $c(k)$ the vector with $c(k)_{i,j}=c_{i,j}$ whenever $i\ge k$ and $c(k)_{i,j}=0$ otherwise. In particular, $c(n)=0$ and $c(1)=c$. The vector $\exp(c(k))u_\la^S$ lies in $U_\la^S(k)$ and we obtain by decreasing induction on $k$ that \[\exp(c(k))u_\la^S=(\exp(c(k))v_{\omega_1}^S)^{\otimes a_1}\otimes\ldots\otimes (\exp(c(k))v_{\omega_{n-1}}^S)^{\otimes a_{n-1}}.\qedhere\] 
\end{proof}

\begin{proof}[Proof of Theorem~\ref{main}] 
In view of Corollary~\ref{degcartcomp} it suffices to prove that $F^S$ coincides with the closure of the set $\{\exp(c)\bm u_\la^S\}\subset \bP(U_\la^S)$ where $\bm u_\la^S$ corresponds to $\bC u_\la^S$. 

Let us write out $\exp(c)u_\la^S$ as in Lemma~\ref{exptensor} and consider the tensor factor $\exp(c)v_{\omega_k}^S$. We may rewrite every $\exp(c_{i,j}\varphi_{i,j})$ as the series $1+\varphi_{i,j}+\frac{\varphi_{i,j}^2}2+\ldots$, expand the product $\exp(c)$ and then retain only those monomials $\varphi^T$ in the result for which $M_T$ is $L_{\om_k}^S$-optimal, since all others act trivially. For a $L_{\om_k}^S$-optimal monomial $M_T$, if $M_Tv_{\om_k}=\pm e_{i_1,\ldots,i_k}$, then $\varphi^T v_{\om_k}^S=\pm e_{i_1,\ldots,i_k}^S$.

Now let us consider $\prod_{i,j} \exp(c_{i,j}f_{i,j})v_{\om_k}\in L_{\om_k}$. Let us expand every $\exp(c_{i,j}f_{i,j})$ as $1+f_{i,j}+\frac{f_{i,j}^2}2+\ldots$, then expand the product and retain only the actions of $L_{\om_k}$-optimal monomials. Then the coordinate of the result corresponding to $e_{i_1,\ldots,i_k}$ will be equal to $\initial_{\grad^A}(C_{i_1,\ldots,i_k})(c)$ where $C_{i_1,\ldots,i_k}$ are the polynomials considered in Theorem~\ref{kernelexp}.

This shows that the coordinate of $\exp(c)v_{\omega_k}^S$ corresponding to $e_{i_1,\ldots,i_k}^S$ will be equal to $\initial_{\grad^A}(C_{i_1,\ldots,i_k})(c)$. If we now compose the embedding $\bP\subset \bP(U_\la)$ from Section~\ref{generalities} with the isomorphism between $\bP(U_\la)$ and $\bP(U_\la^S)$, then we see that $\exp(c)\bm u_\la^S$ lies in $\bP\subset\bP(U_\la^S)$ and its Pl\"ucker coordinates are precisely the values $\initial_{\grad^S}(C_{i_1,\ldots,i_k})(c)$. Finally, we know from Theorem~\ref{kernelexp} that $\initial_{\grad^S}I$ is precisely the ideal of polynomials vanishing in all points with Pl\"ucker coordinates of this form with $c$ ranging over $\bC^{\{1\le i<j\le n\}}$. This concludes the proof.

The obtained embedding of $F^S$ into $\bP(L_\la^S)$ coincides with the one obtained in Proposition~\ref{FSinPLS}, since we have considered the same embedding of $\bP(L_\la^S)$ into $\bP(U_\la)$, the same embedding of $\bP$ into $\bP(U_\la)$ and the same embedding of $F^S$ into $\bP$.
\end{proof} 

\begin{remark}
In the case of abelian PBW degenerations as well as in~\cite{fafefom} and~\cite{favourable} the degenerate flag variety was defined as an orbit closure for a degenerate Lie group. This was then shown to coincide with a certain Gr\"obner degeneration. There is no degenerate group to be seen here, however, it turns out that considering the exponentials of generators of the degenerate algebra is sufficient. The fact that this embedding provided by representation theory coincides with the geometric one given by Proposition~\ref{FSinPLS} shows that we have constructed the ``correct'' embedding.

Now, in these earlier works one could define the degenerate flag variety in complete analogy with Theorem~\ref{main} without mentioning the Lie group formed by the exponentials. Since our degenerate representation theory shares many nice properties with the earlier theories of PBW degenerations (the existence of tensor products, Corollary~\ref{degcartcomp}, Lemma~\ref{exptensor} and, of course, Theorem~\ref{main}), one could argue that a degenerate Lie group is not inherent to a degenerate representation theory but is an additional nice feature of the earlier theories. This idea is strengthened by the below approach which gets rid of the exponentials altogether, requiring only tensor products with a ``Cartan component property'' analogous to Corollary~\ref{degcartcomp}.
\end{remark}

We now present an alternative way of characterizing $F^S$ in terms of the representation theory of $\Phi^n$. First of all, for integral dominant weights $\la$ and $\mu$ consider the tensor product $U_\la^S\otimes U_\mu^S=U_{\la+\mu}^S$. On one hand, by Corollary~\ref{degcartcomp} applied to $\la$ and $\mu$, this product contains $L_\la^S\otimes L_\mu^S$. On the other, it contains $L_{\la+\mu}^S$ as the submodule generated by $u_\la^S\otimes u_\mu^S=u_{\la+\mu}^S$. We obtain an embedding $L_{\la+\mu}^S\subset L_\la^S\otimes L_\mu^S$ and the dual surjection $(L_\la^S)^*\otimes (L_\mu^S)^*\twoheadrightarrow(L_{\la+\mu}^S)^*$. This gives a commutative algebra structure on $\mathcal Q^S=\bigoplus_\la  (L_\la^S)^*$. %This algebra is graded by $\fh^*$ by definition, we denote this grading $\wt$.

\begin{theorem}
$\mathcal Q^S$ is isomorphic to $\mathcal P^S$.
\end{theorem}
\begin{proof}
Both of the algebras decompose as $\bigoplus_\la  (L_\la^S)^*$, we are to identify the multiplicative structures. In view of Corollary~\ref{degcartcomp} the algebra $\mathcal Q^S$ is generated by the components $(L_{\om_i}^S)^*$ and we have a surjection $R\twoheadrightarrow\mathcal Q^S$. We are to show that the kernel of this surjection is $\initial_{\grad^S} I$.

For $\la=(a_1,\ldots,a_{n-1})$ the subspace \[W_\la^S=\Sym^{a_1}(L_{\om_1}^S)\otimes\ldots\otimes\Sym^{a_{n-1}}(L_{\om_{n-1}}^S)\subset U_\la^S\] is seen to be a $\Phi_n$-submodule. When restricted to homogeneity degree $\la$, the kernel of the above surjection is the orthogonal of the submodule $L_\la^S\subset W_\la^S$ generated by $u_\la^S\in W_\la^S$ (the bases in $L_{\om_k}^S$ composed of the $e_{i_1,\ldots,i_k}^S$ provide a duality between $W_\la^S$ and $R_\la$). However, in the proof of Proposition~\ref{dualspaces} it was shown that $L_\la^S\subset W_\la$ is orthogonal to $\initial_{\grad^S} I_\la$ and the theorem follows via the linear isomorphism between $U_\la$ and $U_\la^S$.
%For any monomial $\varphi^T$ acting nontrivially on $u_\la^S$ (i.e. such that $M_T$ is $L_\la$-optimal) Proposition~\ref{phionLS} implies the following. To obtain $\varphi^T u_\la^S$ one must take $M_Tu_\la$, take its maximal $\grad^S$-homogeneous component and take the image of this component under the linear isomorphism between $U_\la$ and $U_\la^S$. The orthogonality with $\initial_{\grad^S} I_\la$ ensues.
\end{proof}

The above theorem can be rephrased as the following characterization of $F^S$ (where $\{(e_{i_1,\ldots,i_k}^S)^*\}\subset (L_\la^S)^*$ is the basis dual to $\{e_{i_1,\ldots,i_k}^S\}$).
\begin{theorem}\label{mainproj}
There exists a surjection $R\twoheadrightarrow\mathcal Q^S$ mapping $X_{i_1,\ldots,i_k}$ to $(e_{i_1,\ldots,i_k}^S)^*$. The zero set of the kernel of this surjection is $F^S\subset\bP$.
\end{theorem} 
We have thus established two ways of characterizing a Gelfand--Tsetlin degeneration in terms of the representation theory of $\Phi_n$. While Theorem~\ref{main} mimics the traditional way of defining PBW degenerations as orbit closures, Theorem~\ref{mainproj} seems more natural and also has the advantage of being independent of the highest weight.

\section{Singular highest weights and partial flag varieties}\label{singular}

The purely representation-theoretic results in the above sections such as Theorem~\ref{basis}, Theorem~\ref{Lfiltration}, Proposition~\ref{filtrationsame} or Corollary~\ref{degenbasis} hold equally well for regular and singular highest weight $\la$. However, results concerned with the geometry of $F^S$ and its defining ideal $I^S$ (Proposition~\ref{FSinPLS}, Theorems~\ref{kerneldet} and~\ref{kernelexp}, Theorem~\ref{toric}, Theorems~\ref{main} and~\ref{mainproj}) are limited to the consideration of the complete flag variety $F$ and its degenerations and, therefore, only deal with regular highest weights. In this section we will recall the necessary facts concerning partial flag varieties and then generalize said results to this setting.

Within this section fix a set ${\bf d}=\{d_1,\ldots,d_l\}\subset\{1,\ldots,n-1\}$ and an integral dominant weight $\la=\sum_j a_{d_j}\om_{d_j}$ with all $a_{d_j}>0$ (i.e.\ having nonzero coordinates precisely at positions $d_1,\ldots,d_l$). The subgroup in $G$ stabilizing $v_\la\in L_\la$ is the standard parabolic subgroup $P_{\bf d}$ (depending only on $\bf d$ and not on the chosen $\la$) and $F_\bd=G/P_\bd$ is the corresponding partial flag variety. Here we obtain $P_\bd=B$ and $F_\bd=F$ when $\bd=\{1,\ldots,n-1\}$, i.e.\ $\la$ is regular.

Consider the subring $R_\bd\subset R$ generated by all Pl\"ucker variables of the from $X_{i_1,\ldots,i_{d_j}}$, this subring is the homogeneous coordinate ring of \[\bP_\bd=\bP(L_{\om_{d_1}})\times\ldots\times\bP(L_{\om_{d_l}}).\] The Pl\"ucker embedding $F_\bd\subset\bP_\bd$ of the partial flag variety is given the ideal $I_\bd=I\cap R_\bd$, denote the homogeneous coordinate ring $\mathcal P_\bd=R_\bd/I_\bd$. The grading $\wt$ restricts to $R_\bd$ as a grading by the semigroup generated by all $\om_{d_j}$. If $\mu$ is a weight in this semigroup, then the corresponding homogeneous components can be identified: $R_{\bd,\mu}=R_\mu$, $I_{\bd,\mu}=I_\mu$ and $\mathcal P_{\bd,\mu}=\mathcal P_\mu$.

The grading $\grad^S$ can also be restricted to $R_\bd$ and we set $\mathcal P_\bd^S=R_\bd/\initial_{\grad^S}I_\bd$. This is the homogeneous coordinate ring of the Gr\"obner degeneration $F_\bd^S\subset\bP_\bd$ of $F_\bd$ given by the ideal $\initial_{\grad^S}I_\bd$. The decreasing filtration on $R$ given by $\grad^S$ induces a filtration on $R_{\bd}$ and, subsequently, on $\mathcal P_{\bd}$. For a weight $\mu$ in the semigroup generated by all $\om_{d_j}$ the obtained filtration on $\mathcal P_{\bd,\mu}$ coincides with that on $\mathcal P_\mu$ and we see (via Proposition~\ref{initdegen}) that the homogeneous components of $\mathcal P_\bd^S$ are associated graded spaces of duals of irreducible representations.

Now, the Segre embedding provides an embedding $\bP_\bd\subset\bP(U_\la)$ and, subsequently, $F_\bd^S\subset\bP(U_\la)$. We have the following generalization of Proposition~\ref{FSinPLS} which is proved in the same manner.
\begin{proposition}
The image of the embedding $F_\bd^S\subset\bP(U_\la)$ is contained in the image of $\bP(L_\la^S)\subset\bP(U_\la)$ (defined in Section~\ref{generalities}).
\end{proposition} 

From the fact that $\initial_{\grad^S}I_\bd=\initial_{\grad^S}I\cap R_\bd$ we immediately obtain generalizations of Theorems~\ref{kerneldet} and~\ref{kernelexp}.
\begin{theorem}\label{kerneldetsing}
$\initial_{\grad^S} I_\bd$ is the kernel of the map from $R_\bd$ to $Q$ sending $X_{i_1,\ldots,i_{d_j}}$ to $\initial_{\grad^A} D_{i_1,\ldots,i_{d_j}}$.
\end{theorem}
\begin{theorem}\label{kernelexpsing}
$\initial_{\grad^S} I_\bd$ is the kernel of the map from $R_\bd$ to $Q$ sending $X_{i_1,\ldots,i_{d_j}}$ to $\initial_{\grad^A}(z_{d_j,d_j}C_{i_1,\ldots,i_{d_j}})$. 
\end{theorem}

Like in the proof of Theorem~\ref{toric}, we see (via Theorem~\ref{kerneldetsing}) that if all inequalities in (a) and (b) are strict, then $\mathcal P_\bd^S$ is the semigroup ring of the semigroup in $\fh^*\oplus\Theta$ generated by the union of all $(\om_{d_j},\Pi_{\om_{d_j}})$. This is precisely the homogeneous coordinate ring of the toric variety associated with the polytope $P_\la$ and we obtain the following generalization.
\begin{theorem}
If all inequalities in (a) and (b) (equivalently, all inequalities in (A) and (B)) are strict, then the Gelfand--Tsetlin degeneration $F_\bd^S$ is the toric variety associated with the polytope $P_\la$. This is isomorphic to the toric variety associated with the Gelfand--Tsetlin polytope $GT_\la$.
\end{theorem}

Finally, we have the generalizations of Theorems~\ref{main} and~\ref{mainproj}. The former is deduced from Theorem~\ref{kernelexpsing} just like Theorem~\ref{main} is deduced from Theorem~\ref{kernelexp} while the latter is a direct consequence of Theorem~\ref{mainproj}.
\begin{theorem}
Let $E_\la$ be the image of $\bC^{\{1\le i<j\le n\}}$ in $\bP(L_\la^S)$ under the map taking $c$ to $\exp(c)\bm v_\la^S$. The Zariski closure of $E_\la$ is the degenerate flag variety $F_\bd^S$. 
\end{theorem}

Let $\mathcal P_\bd^S$ be the subalgebra in $\mathcal P^S$ spanned by $(L_\mu^S)^*$ such that $\mu$ being a sum of the $\om_{d_j}$.
\begin{theorem}
There exists a surjection $R_\bd\twoheadrightarrow\mathcal P_\bd^S$ mapping $X_{i_1,\ldots,i_{d_j}}$ to $(e_{i_1,\ldots,i_{d_j}}^S)^*$ the kernel of which cuts out $F_\bd^S\subset\bP_\bd$.
\end{theorem}

\section{Gr\"obner fans and tropical flag varieties}\label{tropical}

It is evident that all $A$ with properties (a) and (b) (i.e.\ (A) and (B)) form the set $K$ of integer points inside a convex rational polyhedral cone $\mathcal K\subset\Theta^*$. Now, $\sigma$ can be viewed as a linear map from $\Theta^*$ to the space of Gr\"obner degenerations $\Xi=\bR^{\{1\le i_1<\ldots<i_k\le n|1\le k\le n-1\}}$.

\begin{proposition}\label{sigmaK}
The linear map $\sigma$ is injective and each of $\mathcal K$ and $\sigma(\mathcal K)$ is a product of $\bR^{n-1}$ and a simpicial cone of dimension ${n-1}\choose2$. Furthermore, the map $\sigma$ is unimodular in the sense that is establishes a bijection between the set integer points in $\Theta^*$ and the set of integer points in its image.
\end{proposition}
\begin{proof}
The map $\sigma$ can be represented by a $(2^n-2)\times{n\choose2}$-matrix. Choose a pair $1\le i<j\le n$ and consider the row of this matrix corresponding to the tuple $(1,\ldots,i-1,j)$ (i.e.\ the integers between 1 and $i$ with $i$ replaced by $j$). One sees that $T(1,\ldots,i-1,j)_{i,j}=1$ while all other coordinates of $T(1,\ldots,i-1,j)$ are zero. This means that this row in our matrix has exactly one nonzero entry in the column corresponding to the pair $(i,j)$. This shows that the matrix has maximal rank and, therefore, $S$ is injective.

One easily sees that altogether there are ${n-1}\choose2$ inequalities in (a) and (b) and that the functionals on $\Theta^*$ bounded by these inequalities are linearly independent. This immediately implies that $\mathcal K$ has the claimed form. The claim concerning $\sigma(\mathcal K)$ follows from the injectivity of $S$.

The final claim can be obtained as follows. Obviously, the image of any integer point under $\sigma$ is an integer point. Conversely, consider an integer point $S=(s_{i_1,\ldots,i_k})\in \sigma(\Theta^*)$. We claim that $S=\sigma(A)$ where the coordinate $a_{i,j}$ of $A$ is equal to $s_{1,\ldots,i-1,j}$. Indeed, we know that the coordinate $\sigma(A)_{1,\ldots,i-1,j}=a_{i,j}$ and that $S$ is the unique point in the image $\sigma$ with the given coordinates $s_{1,\ldots,i-1,j}$, since all other coordinates are expressed as linear combinations of these.
\end{proof}

Let us briefly introduce the Gr\"obner fan and the tropicalization of the flag variety $F$, the details can be found in~\cite[Chapter 3]{MLS}. Every point in $S=(s_{i_1,\ldots,i_k})\in\Xi$ defines a Gr\"obner degeneration of $F$ but, as mentioned in Remark~\ref{projlim}, increasing all $s_{i_1,\ldots,i_k}$ for a chosen $k$ by the same constant does not change the degeneration. Therefore we can restrict our attention to the subspace $\Xi_0$ in which $s_{1,\ldots,k}=0$ for all $k$. Note that $\sigma(\mathcal K)\subset\Xi_0$. 

Let us define an equivalence relation on $\Xi_0$ by setting $S\sim S'$ if and only if $\initial_{\grad^{S'}}I=\initial_{\grad^S}I$. Each equivalence class is the relative interior of a closed convex rational polyhedral cone in $\Xi_0$. Together all these cones form a complete fan in $\Xi_0$ known as the Gr\"obner fan of the variety $F$. Let us consider all the cones in this fan such that for a point $S$ in the relative interior of the cone the initial ideal $\initial_{\grad^{S}}I$ does not contain any monomials in the Pl\"ucker variables. These cones form a subfan in the Gr\"obner fan that is the tropicalization of $F$ with respect to the identically zero valuation on $\bC$. This subfan is also known as the {\it tropical flag variety}.

\begin{theorem}\label{maxcone}
$\sigma(\mathcal K)$ is a cone in the Gr\"obner fan of $F$. Moreover, $\sigma(\mathcal K)$ is a maximal cone in the tropicalization of $F$.
\end{theorem}
\begin{proof}
For every $A\in K$ we know from Theorem~\ref{kerneldet} that $\mathcal P^{\sigma(A)}$ can be embedded into a polynomial ring, therefore $\initial_{\grad^{\sigma(A)}}I$ is prime and hence monomial free. From Proposition~\ref{sigmaK} we now deduce that all the integer points in $\sigma(\mathcal K)$ are contained in (the support of) the tropical flag variety and, consequently, so is all of $\sigma(\mathcal K)$ since it is a rational cone. Furthermore, as shown in~\cite[Lemma 3.2.10]{MLS}, a maximal cone in the tropicalization of $F$ can have dimension at most $\dim F={n\choose 2}$ which is precisely the dimension of $\sigma(\mathcal K)$. We see that it suffices to prove the first claim and the second will follow.

We know that for every integer point $S$ in the relative interior of $\sigma(\mathcal K)$ the initial ideal $\initial_{\grad^S}I$ is the toric ideal $J$ discussed in Theorem~\ref{toric}, hence the same holds for every (not necessarily integer) point in the relative interior. To prove that $\sigma(\mathcal K)$ is a cone in the Gr\"obner fan we are to show that for every point $S$ in its relative boundary the ideal $\initial_{\grad^S}I$ differs from $J$. Again, since $\sigma(\mathcal K)$ and all of its proper faces are rational cones, it suffices to prove the last assertion for integer points $S$, i.e.\ points of the form $\sigma(A)$ where $A=(a_{i,j})\in K$ is such that at least one of the inequalities in (a) and (b) is an equality.

Now,~\cite[Theorem 14.16]{MS} provides an explicit description of $J$. It is generated by binomials of the form 
\begin{equation}\label{Jrelation}
X_{i_1,\ldots,i_k}X_{j_1,\ldots,j_l}-X_{\max(i_1,j_1),\ldots,\max(i_k,j_k)}X_{\min(i_1,j_1),\ldots,\min(i_k,j_k),j_{k+1},\ldots,j_l}
\end{equation}
where $k\le l$. 

Choose an $A=(a_{i,j})\in K$. Suppose that we have $a_{i,i+1}+a_{i+1,i+2}=a_{i,i+2}$ for some $1\le i\le n-2$. Let us show that $\initial_{\grad^{\sigma(A)}}I$ differs from $J$ by presenting a binomial of the form~(\ref{Jrelation}) which is not contained in $\initial_{\grad^{\sigma(A)}}I$. Indeed, we have 
\begin{align*}
&\initial_{\grad^A}(D_{1,\ldots,i})=z_{1,1}\ldots z_{i,i},\\
&\initial_{\grad^A}(D_{1,\ldots,i-1,i+1,i+2})=z_{1,1}\ldots z_{i-1,i-1}z_{i,i+1}z_{i+1,i+2}+z_{1,1}\ldots z_{i-1,i-1}z_{i,i+2}z_{i+1,i+1},\\
&\initial_{\grad^A}(D_{1,\ldots,i-1,i+1})=z_{1,1}\ldots z_{i-1,i-1}z_{i,i+1}\text{ and}\\
&\initial_{\grad^A}(D_{1,\ldots,i,i+2})=z_{1,1}\ldots z_{i,i}z_{i+1,i+2}.
\end{align*}
Theorem~\ref{kerneldet} is now seen to imply \[X_{1,\ldots,i}X_{1,\ldots,i-1,i+1,i+2}-X_{1,\ldots,i-1,i+1}X_{1,\ldots,i,i+2}\notin \initial_{\grad^{\sigma(A)}}I.\]

Now, suppose that $a_{i,j}+a_{i+1,j+1}=a_{i,j+1}+a_{i+1,j}$ for some $1\le i<j-1\le n-2$. Similarly to the above we observe that in this case \[X_{1,\ldots,i}X_{1,\ldots,i-1,j,j+1}-X_{1,\ldots,i-1,j}X_{1,\ldots,i,j+1}\notin \initial_{\grad^{\sigma(A)}}I.\qedhere\]
\end{proof}

Let us stress that, in view of the above theorem, the relative interior of $\sigma(\mathcal K)$ is the set of {\it all} Gr\"obner degenerations such that $\initial_{\grad^S}I=J$. General properties of Gr\"obner fans found in~\cite[Section 3.3]{MLS} can now be used to obtain the following.
\begin{cor}
The degeneration $F^{\sigma(A)}$ depends only on the minimal face of $\mathcal K$ containing $A$. Furthermore, if $A,B\in K$ are such that the minimal face of $\mathcal K$ containing $A$ contains the minimal face of $\mathcal K$ containing $B$, then $F^{\sigma(A)}$ is a Gr\"obner degeneration of $F^{\sigma(B)}$.
\end{cor}

\begin{remark}
The toric degeneration is seen to be a Gr\"obner degeneration of any other Gelfand--Tsetlin degeneration. This allows us to use general properties of Gr\"obner degenerations and initial ideals to generalize various facts known about the toric degeneration to all GT degenerations. For instance, one can now easily deduce that any of the ideals $\initial_{\grad^{\sigma(A)}}I$ with $A\in K$ is generated by its quadratic part. Or that the set of all monomials $X_{i_1^1,\ldots,i_{k_1}^1}\ldots X_{i_1^N,\ldots,i_{k_N}^N}\in R$ such that the tuples $(i_1^1,\ldots,i_{k_1}^1),\ldots,(i_1^N,\ldots,i_{k_N}^N)$ are the columns of a semistandard Young tableau projects to a basis in $\mathcal P^{\sigma(A)}$ (see~\cite[Corollary 14.9]{MS}). 
\end{remark}

Without going into specifics let us explain how Theorem~\ref{maxcone} can be generalized to partial flag varieties. For $\bd\subset\{1,\dots,n-1\}$ one can consider the subspace $\Xi_\bd\subset\Xi$ consisting of $S=(s_{i_1,\ldots,i_k})$ with $s_{i_1,\ldots,i_k}=0$ whenever $k\notin\bd$. We then define the map $\sigma_\bd:\Theta^*\to\Xi_\bd$ as the composition of $\sigma$ and coordinatewise projection onto $\Xi_\bd$. The assertion is that $\sigma_\bd(\mathcal K)$ is a maximal cone in the tropicalization of $F_\bd$, the latter being contained in $\Xi_\bd\cap\Xi_0$. This can first be proved for Grassmannians by induction on $n$ and then generalized to arbitrary $\bd$ by considering a point in the relative boundary of $\sigma_\bd(\mathcal K)$ and showing that it projects into the relative boundary of some $\sigma_{\{i\}}(\mathcal K)$ with $i\in\bd$.

To complete the section let us give a fully explicit description of the maximal cone $\sigma(K)$ in the tropicalization, i.e. its minimal H-description.
\begin{proposition}
The cone $\sigma(\mathcal K)$ consists of such $S=(s_{i_1,\ldots,i_k})\in\Xi$ that all \[s_{i_1,\ldots,i_k}=\sum_{i,j} T(i_1,\ldots,i_k)_{i,j}s_{1,\ldots,i-1,j}\] and
\begin{enumerate}[label=({\alph*}$'$)]
\item $s_{1,\ldots,i-1,i+1}+s_{1,\ldots,i,i+2}\le s_{1,\ldots,i-1,i+2}$ for any $1\le i\le n-2$,
\item $s_{1,\ldots,i-1,j}+s_{1,\ldots,i,j+1}\le s_{1,\ldots,i-1,j+1}+s_{1,\ldots,i,j}$ for any $1\le i<j-1\le n-2$.
\end{enumerate}
\end{proposition}
\begin{proof}
The inequalities in (a) and (b) provide a minimal H-description of the cone $\mathcal K$. One then applies the map $\sigma$ to these inequalities as described in the proof of Proposition~\ref{sigmaK} to obtain the proposition.
\end{proof}

\begin{remark}
The maximal cone in the tropicalization corresponding to the toric initial ideal $J$ is also discussed in~\cite[Example 7.3]{KaM}. There a redundant H-description of this cone is given.
\end{remark}

\section{The dual construction}\label{dual}

The results in Sections \ref{bases}--\ref{tropical} can be dualized via the Dynkin diagram automorphism for type $\mathrm{A}_{n-1}$. Let us show how this dualization works and why most of it, in a sense, reduces to the results already obtained.

For a tuple $1\le i_1<\ldots<i_k\le n$ let $j_{k+1}<\ldots<j_n$ be the elements of $\{1,\ldots,n-1\}\backslash\{i_1,\ldots,i_k\}$. We define the points $\widetilde T(i_1,\ldots,i_k)\in\Theta$ by 
\begin{equation}
\widetilde T(i_1,\ldots,i_k)_{l,m}=
\begin{cases}
1\text{ if }m\ge k+1\text{ and }l=j_m,\\
0\text{ otherwise}. 
\end{cases}
\end{equation}
In other words, we have the coordinate corresponding to pair $(j_m,m)$ equal to $1$ for every $k+1\le m\le n$ with $j_m<m$ and all other coordinates zero. We then define $\widetilde\Pi_{\om_k}$ as the set of all $\widetilde T(i_1,\ldots,i_k)$ and $\widetilde\Pi_\la$ as the corresponding Minkowski sum. Let $\widetilde P_\la$ be the convex hull of $\widetilde\Pi_\la$.

Consider the involution $\eta$ of $\Theta$ with $\eta(T)_{i,j}=T_{n+1-j,n+1-i}$. In terms of Example~\ref{sl3polytopes} this is simply reflection across a vertical line.
\begin{proposition}
$\widetilde P_\la$ is unimodularly equivalent to $GT_{\la}$.
\end{proposition}
\begin{proof}
One sees that in the above notation we have \[\widetilde T(i_1,\ldots,i_k)=\eta(T(n+1-j_n,\ldots,n+1-j_{k+1})).\] Hence, $\widetilde\Pi_{\om_k}=\eta(\Pi_{\om_{n-k}})$ and $\widetilde\Pi_\la=\eta(\Pi_{\widetilde\la})$ where $\widetilde\la$ is the image of $\la$ under the linear involution of $\fh^*$ that transposes $\om_k$ and $\om_{n-k}$. We see that $\eta(\widetilde P_\la)=P_{\widetilde \la}$, i.e.\ $\widetilde P_\la$ is unimodularly equivalent to $P_{\widetilde \la}$ and hence $GT_{\widetilde \la}$. However, $GT_{\widetilde \la}$ is easily seen to be unimodularly equivalent to $GT_\la$.
\end{proof}

Now consider the linear involution $\zeta$ of $\fn_-$ that maps $f_{i,j}$ to $-f_{n+1-j,n+1-i}$, this is a Lie algebra automorphism. The representations $L_\la$ and $L_{\widetilde \la}$ are conjugate under this automorphism, meaning that there exists a linear isomorphism $\zeta_\la:L_\la\to L_{\widetilde \la}$ such that $f\zeta_\la(v)=\zeta_\la(\zeta(f)v)$ for any $f\in\fn_-$ and $v\in L_\la$. These representations are also contragradient duals of each other but we will not be making direct use of this here.
\begin{proposition}\label{dualbasis}
The set $\{\prod f_{i,j}^{T_{i,j}}v_{\widetilde\la},T\in\Pi_{\widetilde\la}\}\subset L_{\widetilde\la}$ with the products ordered by $j$ decreasing from left to right is, up to signs, the image of the set $\{M_Tv_\la,T\in\Pi_\la\}$ under $\zeta_\la$. In particular the former set constitutes a basis in $L_{\widetilde\la}$.
\end{proposition}
\begin{proof}
Since $v_\la$ and $v_{\widetilde\la}$ are the only highest weight vectors in the respective representations up to a scalar factor, we can assume that $\zeta_\la(v_\la)=v_{\widetilde \la}$. We see that, in view of the definitions of $\eta$ and $\zeta_\la$, for $T\in\Pi_\la$ the image $\zeta_\la(M_Tv_\la)$ is $\pm\prod f_{i,j}^{\eta(T)_{i,j}}v_{\widetilde\la}$. It remains to recall that $\eta(\Pi_\la)=\Pi_{\widetilde\la}$.
\end{proof}

To dualize the results in Section~\ref{gtdegens} one considers $A\in\Theta^*$ satisfying the same inequalities (a) and (b) as before and sets \[\widetilde\sigma(A)_{i_1,\ldots,i_k}=A(\widetilde T(i_1,\ldots,i_k)).\] Let $\eta^*$ be the involution of $\Theta^*$ dual to $\eta$, i.e.\ given by $\eta^*(A)_{i,j}=A_{n+1-j,n+1-i}$. We see that for a monomial $M\in\mU(\fn_-)$ we have $\deg^A M=\deg^{\eta^*(A)}\zeta(M)$ where we extend $\zeta$ to the universal enveloping algebra. This provides dual versions of Lemma~\ref{minmon} and Theorem~\ref{Lfiltration} via the conjugation between $L_\la$ and $L_{\widetilde\la}$, we omit the details. 

Furthermore, let $\Upsilon$ be the involution of $R$ mapping $X_{i_1,\ldots,i_k}$ to $X_{n+1-j_n,\ldots,n+1-j_{k+1}}$ where again \[\{j_{k+1},\ldots,j_n\}=\{1,\ldots,n-1\}\backslash\{i_1,\ldots,i_k\}.\] 
\begin{proposition}
The ideals $\initial_{\grad^{\widetilde\sigma(A)}}I$ and $\Upsilon(\initial_{\grad^{\sigma(\eta^*(A))}}I)$ coincide.
\end{proposition}
\begin{proof}
This follows from $\Upsilon(I)=I$ and $\grad^{\widetilde\sigma(A)}X_{i_1,\ldots,i_k}=\grad^{\sigma(\eta^*(A))}\Upsilon(X_{i_1,\ldots,i_k})$.
\end{proof}
The dual versions of the results concerned with the Gr\"obner degeneration $F^{\widetilde\sigma(A)}$ in Sections~\ref{gtdegens} and~\ref{singular} are now obtained straightforwardly and we, again, do not go into details. 

We point out, however, that in the dual version of Theorem~\ref{toric} the initial ideal obtained when all inequalities in (a) and (b) are strict will not be the toric ideal $J$, instead we will have $\initial_{\grad^{\widetilde\sigma(A)}I}=\Upsilon(J)$. The variety $F^{\widetilde\sigma(A)}$ will again be the toric variety of the GT polytope but $\Upsilon(J)$ provides a different projective embedding thereof. This means that when dualizing the results in Section~\ref{tropical} we obtain a different maximal cone in the tropical flag variety:
\begin{theorem}
The cone $\widetilde\sigma(K)$ is a maximal cone in the tropicalization of $F$ which is different from $\sigma(K)$ when $n\ge 4$.
\end{theorem}
Thus we have explicit descriptions of two different series of maximal cones in tropical flag varieties. This pair of cones is transposed by the action of $\bZ/2\bZ$ on the tropical flag variety, see~\cite{BLMM}. This is in contrast with the maximal cone in in the tropicalization that was obtained in~\cite{fafefom}, since the latter is invariant under the $\bZ/2\bZ$-action.

To dualize the results in Section~\ref{degenact} one considers the associative algebra $\widetilde\Phi_n$ with generators $\widetilde\varphi_{i,j},1\le i<j\le n$ and relations $\widetilde\varphi_{i_1,j_1}\widetilde\varphi_{i_2,j_2}=0$ whenever $j_1<j_2$ and $\widetilde\varphi_{i_1,j}\widetilde\varphi_{i_2,j}=\widetilde\varphi_{i_2,j}\widetilde\varphi_{i_1,j}$ for all $1\le i_1<i_2<j\le n$. This algebra acts on $L_\la$ by $\widetilde\varphi_{i,j}$ acting like $f_{i,j}$ on $\mU(\bigoplus_{m\le j}\bC f_{l,m})v_\la$ and annihilating all weight vectors outside of this space. 

There is an isomorphism $\Psi$ between $\Phi_n$ and $\widetilde\Phi_n$ mapping $\varphi_{i,j}$ to $-\widetilde\varphi_{n+1-j,n+1-i}$.
\begin{proposition}
Define another action of $\Phi_n$ on the space $L_{\la}$ by letting $\varphi_{i,j}$ act as $\Psi(\varphi_{i,j})$ in the above $\widetilde\Phi_n$-action. The obtained $\Phi_n$-module is isomorphic to $L_{\widetilde\la}$ with the action considered in Section~\ref{degenact}.
\end{proposition}
\begin{proof}
The isomorphism is given by the map $\zeta_\la$. Indeed, for $2\le j\le n$ the involution $\zeta$ maps $\bigoplus_{m\le j}\bC f_{l,m}$ bijectively onto $\fn_-(n+1-j)$ and, therefore, $\zeta_\la$ maps $\mU(\bigoplus_{m\le j}\bC f_{l,m})v_{\la}$ bijectively onto $L_{\widetilde\la}(n+1-j)$. We now see that for a weight vector $v\in L_\la$ if $v\in\mU(\bigoplus_{m\le j}\bC f_{l,m})v_{\la}$, then \[\zeta_\la(\Psi(\varphi_{n+1-j,n+1-i})v)=-\zeta_\la(f_{i,j}v)=f_{n+1-j,n+1-i}\zeta_\la(v)=\varphi_{n+1-j,n+1-i}\zeta_\la(v),\] and if $v\notin\mU(\bigoplus_{m\le j}\bC f_{l,m})v_{\la}$, then \[\zeta_\la(\Psi(\varphi_{n+1-j,n+1-i})v)=0=\varphi_{n+1-j,n+1-i}\zeta_\la(v).\qedhere\]
\end{proof}
In other words, $\zeta_\la$ establishes an isomorphism between the $\Phi_n$-module $L_\la$ and the $\widetilde\Phi_n$-module $L_{\widetilde\la}$ modulo the isomorphism $\Psi$. Further details regarding the dualization of results concerned with the action of $\Phi_n$ are now recovered straightforwardly.

\section{Addendum: approach via non-abelian gradings}\label{addendum}

In this section we propose an alternative solution to the original problem of realizing the Gelfand--Tsetlin toric variety in a context of degenerate representation theory. Let us first recall the following standard definitions concerning initial ideals in free associative algebras. These notions originate from~\cite{Be}. 

Let $\mathcal F_n$ be the free associative $\bC$-algebra generated by the symbols $\hat f_{i,j}$ with $1\le i<j\le n$. We have the surjection $\mathcal F_n\twoheadrightarrow\mU(\fn_-)$ mapping $\hat f_{i,j}$ to $f_{i,j}$, denote $\mathcal I$ the kernel of this surjection. 

Now consider the set of monomials (i.e. products of $\hat f_{i,j}$) in $\mathcal F_n$. Under multiplication these monomials form the free monoid $\{f_{i,j}\}^*$, let $\prec$ be a total ordering of this monoid, i.e. a total order on the monomials such that for any monomials $x$, $y$ and $z$ the condition $x\prec y$ implies $xz\prec yz$ and $zx\prec zy$. Then for any $f\in\mathcal F_n$ we may define its initial part $\initial_\prec f$ as the $\prec$-\emph{maximal} monomial occurring in $f$. The initial (two-sided) ideal $\initial_{\prec}\mathcal I$ is then spanned by the monomials $\initial_\prec f$ with $f$ ranging over $\mathcal I$.

Furthermore, for an integral dominant weight $\la$ let $\mathcal I_\la\subset\mathcal F_n$ be the left ideal annihilating $L_\la$. Then $\initial_\prec\mathcal I_\la$ defined as the span of the initial parts of elements of $\mathcal I_\la$ will be a left ideal containing $\initial_{\prec}\mathcal I$. In other words, we obtain a left module $L_\la^\prec=\mathcal F_n/\initial_\prec\mathcal I_\la$ over $\mU^\prec=\mathcal F_n/\initial_{\prec}\mathcal I$. This module is generated by the vector $v_\la^\prec$, the image of $1$.

A key property of the algebra $\mU^\prec$ is that it is graded by the non-abelian monoid $\{f_{i,j}\}^*$ in a way that respects the non-commutative multiplication. The modules $L_\la^\prec$ are also $\{f_{i,j}\}^*$-graded in a way that respects the left $\mU^\prec$-action. There is another way of defining these graded objects: as associated graded spaces.

The order $\prec$ on $\{f_{i,j}\}^*$ defines a filtration of $\mU(\fn_-)$ by this totally ordered set. For $x\in \{f_{i,j}\}^*$ the filtration component $\mU(\fn_-)_{\preceq x}$ is defined as the span of all PBW monomials $f_{i_1,j_1}\ldots f_{i_m,j_m}$ such that $\hat f_{i_1,j_1}\ldots \hat f_{i_m,j_m}\preceq x$. We also define the space $\mU(\fn_-)_{\prec x}$ as the span of all PBW monomials $f_{i_1,j_1}\ldots f_{i_m,j_m}$ such that $\hat f_{i_1,j_1}\ldots \hat f_{i_m,j_m}\prec x$. We then have the associated $\{f_{i,j}\}^*$-graded space \[\gr\mathcal U(\fn_-)=\bigoplus_{x\in\{f_{i,j}\}^*} \mU(\fn_-)_{\preceq x}/\mU(\fn_-)_{\prec x}.\] In view of the condition on $\prec$ the space $\gr\mathcal U(\fn_-)$ inherits a multiplicative structure from $\mU(\fn_-)$ and is, in fact, an associated graded algebra.

Further, the $\{f_{i,j}\}^*$-filtration on $\mU(\fn_-)$ induces a $\{f_{i,j}\}^*$-filtration on every $L_\la$ by acting on $v_\la$ and we may again consider the associated $\{f_{i,j}\}^*$-graded spaces $\gr L_\la$. The following fact is proved in complete analogy with Proposition~\ref{initdegen}.
\begin{proposition}\label{adinitdegen}
$\gr\mU(\fn_-)$ and $\mU^\prec$ are isomorphic as $\{f_{i,j}\}^*$-graded algebras while $\gr L_\la$ and $L_\la^\prec$ are isomorphic as graded modules over these algebras.
\end{proposition}

\begin{remark}\label{notmonomial}
To obtain abelian PBW degenerations as well as the degenerations from~\cite{favourable} and~\cite{fafefom} one needs to generalize this setting. Namely, one needs to consider total orderings of an arbitrary semigroup $\Delta$ equipped with a homomorphism $\{f_{i,j}\}^*\to\Delta$ (in the mentioned cases $\Delta$ is either $\bZ$ or $\bZ^{n(n-1)/2}$). One then obtains a $\Delta$-grading on $\mathcal F_n$ and the total order on $\Delta$ induces a partial monomial order on $\mathcal F_n$. The key novelty of the present construction is that we consider a non-abelian $\Delta$ (namely, all of $\{f_{i,j}\}^*$).
\end{remark}

We now specialize to a particular order $\prec$ which is defined as follows. For a monomial $x=\hat f_{i_1,j_1}\ldots\hat f_{i_m,j_m}$ define $|x|=\sum_k (j_k-i_k)$. Now consider two monomials $x=\hat f_{i_1,j_1}\ldots\hat f_{i_m,j_m}$ and $y=\hat f_{p_1,q_1}\ldots\hat f_{p_r,q_r}$. First, we set $x\prec y$ whenever $|x|<|y|$. Now, if $|x|=|y|$ we compare the monomials lexicographically. Namely we consider the least such $k$ that $(i_k,j_k)\neq(p_k,q_k)$ and set $x\prec y$ whenever $i_k<p_k$ or $i_k=p_k$ and $j_k<q_k$. Note that $|x|=|y|$ ensures that neither of $x$ and $y$ is a prefix of the other. 

\begin{remark}\label{wthomogeneous}
We introduce the function $|x|$ rather than simply comparing monomials lexicographically in order to avoid indefinite situations in which one monomial is a prefix of the other. However, we could consider any other increasing function $g$ on $[1,n]$ and set $|x|=\sum_k (g(j_k)-g(i_k))$ instead. It is easily seen that this would not alter the initial ideals $\initial_\prec\mathcal I$ and $\initial_\prec\mathcal I_\la$. This is a special case of the following general principle: if an ideal is homogeneous with respect to some grading, then its initial ideal is determined by the order relations between pairs of monomials of the same grading. In our case the ideals are homogeneous with respect to weight, i.e. the grading $\wt(\hat f_{i,j})=\alpha_{i,j}$.
\end{remark}

The algebra $\mU^\prec$ is easy to describe. Denote $\chi_{i,j}\in\mU^\prec$ the image of $\hat f_{i,j}$, the $\chi_{i,j}$ generate $\mU^\prec$. 

\begin{proposition}\label{nonzeromons}
A product $\chi_{i_1,j_1}\ldots\chi_{i_m,j_m}$ is nonzero if and only if $i_1\le\ldots\le i_m$ and $j_{k}\le j_{k+1}$ whenever $i_k=i_{k+1}$. These nonzero products form a basis in $\mU^\prec$.
\end{proposition}
\begin{proof}
Consider a monomial $x=\hat f_{i_1,j_1}\ldots\hat f_{i_m,j_m}$ such that for some $k$ either $i_k>i_{k+1}$ or $i_k=i_{k+1}$ and $j_k>j_{k+1}$. Let $y$ be obtained from $x$ by exchanging $f_{i_k,j_k}$ and $f_{i_{k+1},j_{k+1}}$. The ideal $\mathcal I$ contains an element of the form $x-y-z$ such that either $z=0$ or $z\prec x$ (due to the commutation relations in $\fn_-$). Since we also have $y\prec x$, we see that $x\in\initial_\prec\mathcal I$. The fact that such $x$ span $\initial_\prec\mathcal I$ follows from $\initial_\prec\mathcal I$ having the same character as $\mathcal I$ (i.e. the same dimensions of $\wt$-graded components in terms of Remark~\ref{wthomogeneous}).
\end{proof}

For $T\in\bZ_{\ge 0}^{\{1\le i<j\le n\}}$ denote $\chi^T=\prod_{i,j} \chi_{i,j}^{T_{i,j}}$ with the factors ordered so that the product is nonzero. The following fact explains the usefulness of the order $\prec$.
\begin{lemma}\label{annideal}
For any integral dominant weight $\la$ the $\mU^\prec$-module $L_\la^\prec$ is isomorphic to $\mU^\prec/\mathcal J_\la$ where $\mathcal J_\la$ is the left ideal spanned by $\chi^T$ with $T\notin\Pi_\la$.
\end{lemma}
\begin{proof}
This can be proved using the notion of essential signatures (see~\cite{favourable,Go,essential,MY}). Namely, choose a PBW basis $\mathcal B$ in $\mU(\fn_-)$ and a bijection $\theta:\mathcal B\to \bZ_{\ge 0}^{\{1\le i<j\le n\}}$ taking each basis element to its exponent vector. Consider a total order $\prec'$ of the semigroup $\bZ_{\ge 0}^{\{1\le i<j\le n\}}$. Then for a dominant integral weight $\la$ the set of essential signatures $\essential(\la)\subset \bZ_{\ge 0}^{\{1\le i<j\le n\}}$ (with respect to the choice of $\mathcal B$ and $\prec'$) consists of such $T$ that $\theta^{-1}(T)v_\la$ is not in the linear span of all $\theta^{-1}(T')v_\la$ with $T'\prec' T$. Obviously, the vectors $\theta^{-1}(T)v_\la$ with $T\in\essential(\la)$ form a basis in $L_\la$. 

For $T\in\bZ_{\ge 0}^{\{1\le i<j\le n\}}$ denote $\hat f^T=\prod_{i,j} \hat f_{i,j}^{T_{i,j}}$ with the factors ordered so that the image of this product in $\mU^\prec$ is nonzero. We may now set $T_1\prec T_2$ whenever $\hat f^{T_1}\prec\hat f^{T_2}$ to define a total order on the semigroup $\bZ_{\ge 0}^{\{1\le i<j\le n\}}$. Also note that the image of $\hat f^T$ in $\mU(\fn_-)$ is the monomial $M_T$ considered in the previous sections. %Consider the PBW basis in $\mU(\fn_-)$ consisting of the images of the $\hat f^T$ (this coincides with the PBW basis consisting of ordered monomials considered in the previous sections). 

Now, since the ideal $\initial_\prec\mathcal I_\la$ is monomial, we see that $L_\la^\prec=\mU^\prec/\tilde{\mathcal J}_\la$ where $\tilde{\mathcal J}_\la$ (the annihilator of $v_\la^\prec$) is spanned by some set of the $\chi^T$. The characterization of $L_\la^\prec$ as of an associated graded space given by Proposition~\ref{adinitdegen} implies that the $\chi^T$ that are not contained in $\tilde{\mathcal J_\la}$ compose the set $\essential(\la)$ of essential signatures with respect to the chosen PBW basis and $\prec$. Since we are to prove that $\tilde{\mathcal J_\la}=\mathcal J_\la$, we are to show that $\essential(\la)=\Pi_\la$.

It is known that $\essential(\la+\mu)$ contains the Minkowski sum $\essential(\la)+\essential(\mu)$ (see~\cite[Proposition 2]{Go}). Therefore, it suffices to prove that $\essential(\om_k)=\Pi_{\om_k}$ for all $1\le k\le n-1$. This amounts to showing that whenever $M_T v_{\om_k}=\pm e_{i_1,\ldots,i_k}$, we have $T\succeq T(i_1,\ldots,i_k)$. In the proof of Lemma~\ref{minmon} it was established that if $M_T v_{\om_k}=\pm e_{i_1,\ldots,i_k}$, then $M_{T(i_1,\ldots,i_k)}$ can be obtained from $M_T$ by a series of operations each of which either replaces $f_{l,m}f_{i,j}$ with $f_{l,j}f_{i,m}$ for some $l<i<j<m$ or replaces $f_{i,l}$ with $f_{i,j}f_{j,l}$ for some $i<j<l$. One sees that both of these operations decrease the exponent vector of the monomial with respect to $\prec$ and the theorem ensues.
\end{proof}

\begin{remark}
There are other ways of interpreting $\Pi_\la$ as a set of essential signatures. Proposition~\ref{minmonstrict} implies that when all inequalities in (A) and (B) are strict, the $M_T$ with $T\in\Pi_\la$ are the only ordered monomials that are $L_\la$-optimal. This means that in this case $\Pi_\la$ is the set of essential signatures with respect to, again, the PBW basis consisting of the monomials $M_T$ and the degree lexicographic order given by degree $A(T)$ and any lexicographic order. Furthermore, \cite{MY} describes a whole family of orders with respect to which (and yet again the same PBW basis) $\Pi_\la$ is the set of essential signatures. This family does not seem to contain the orders considered here, however.
\end{remark}

Lemma~\ref{annideal} shows that the structure of $L_\la^\prec$ is very simple. The only $\chi^T$ acting nontrivially on $v_\la^\prec$ are those with $T\in\Pi_\la$ and $\{\chi^Tv_\la^\prec,T\in\Pi_\la\}$ is a basis in $L_\la^\prec$. This allows us to immediately obtain an analog of Theorem~\ref{main} without defining tensor products (which will then be introduced to give an analog of Theorem~\ref{mainproj}). 

For a complex vector $c=(c_{i,j})\in\bC^{\{1\le i<j\le n\}}$ and any $\la$ define the operator $\exp(c)$ on $L_\la^\prec$ as the product \[\exp(c_{1,2}\chi_{1,2})\dots\exp(c_{1,n}\chi_{1,n})\exp(c_{2,3}\chi_{2,3})\ldots\exp(c_{n-1,n}\chi_{n-1,n})\] (with the factors again ordered first by $i$ and then by $j$). These operators are invertible and induce automorphisms of $\bP(L_\la^\prec)$ which we also denote $\exp(c)$. We denote $\bm v_\la^\prec\in\bP(L_\la^\prec)$ the point corresponding to $v_\la^\prec$.
\begin{theorem}
For an integral dominant $\la$ the closure of the set of points $\exp(c)\bm v_\la^\prec\in\bP(L_\la^\prec)$ with $c$ ranging over $\bC^{\{1\le i<j\le n\}}$ is isomorphic to the toric variety associated with $P_\la$.
\end{theorem}
\begin{proof}
Consider the homogeneous coordinates in $\bP(L_\la^\prec)$ given by the basis $\{\chi^Tv_\la^\prec,T\in\Pi_\la\}$. We see that the homogeneous coordinate of $\exp(c)\bm v_\la^\prec$ corresponding to $\chi^Tv_\la^\prec$ is equal to $b_T\prod_{i,j} c_{i,j}^{T_{i,j}}$ for a certain constant $b_T$ independent of $c$. This means that by scaling the chosen basis we can make the homogeneous coordinate of $\exp(c)\bm v_\la^\prec$ corresponding to $\chi^Tv_\la^\prec$ simply equal to $\prod_{i,j} c_{i,j}^{T_{i,j}}$. The closure of the set of points with such homogeneous coordinates is precisely the desired toric variety.
\end{proof}

We move on to defining tensor products. Note that $\mU^\prec$ and all $L_\la^\prec$ are graded by $\bZ_{\ge 0}^{\{1\le i<j\le n\}}$ via $\grad(\chi^T)=T$. We consider the category $\mathcal D$ of finite-dimensional $\bZ_{\ge 0}^{\{1\le i<j\le n\}}$-graded $\mU^\prec$-modules $L$ with the following property. For any $\chi_{i,j}$ and any $\grad$-homogeneous $v\in L$ we have $\chi_{i,j}v=0$ whenever $\grad(v)_{l,m}\neq 0$ for some $(l,m)$ with $l<i$ or $l=i$ and $m<j$ (in other words, $\hat f^{\grad(v)}\prec\hat f_{i,j}$). Evidently, the $L_\la^\prec$ lie in this category.

For $L_1$ and $L_2$ in $\mathcal D$ we see that a $\bZ_{\ge 0}^{\{1\le i<j\le n\}}$-grading $\grad$ is induced on $L_1\otimes L_2$. For a $\grad$-homogeneous vector $v_1\otimes v_2\in L_1\otimes L_2$ we set $\chi_{i,j}(v_1\otimes v_2)=0$ if $\hat f^{\grad(v_1\otimes v_2)}\prec\hat f_{i,j}$, otherwise we set $\chi_{i,j}(v_1\otimes v_2)=\chi_{i,j}(v_1)\otimes v_2+v_1\otimes \chi_{i,j}(v_2)$.
\begin{proposition}
This makes $L_1\otimes L_2$ a $\mU^\prec$-module lying in $\mathcal D$. The defined tensor product in $\mathcal D$ is associative and symmetric.
\end{proposition}
\begin{proof}
We see that whenever $l<i$ or $l=i$ and $m<j$ the image $\chi_{l,m} L$ is contained in the span of vectors $v$ with $\hat f^{\grad(v)}\prec\hat f_{i,j}$ and therefore  $\chi_{i,j}\chi_{l,m} L=0$. The remaining assertions are straightforward.
\end{proof}

The Cartan components are easily found.
\begin{proposition}
For integral dominant $\la$ and $\mu$ the $\mU^\prec$-submodule in $L_\la^\prec\otimes L_\mu^\prec$ generated by $v_\la^\prec\otimes v_\mu^\prec$ is isomorphic to $L_{\la+\mu}^\prec$.
\end{proposition}
\begin{proof}
$\chi^T(v_\la^\prec\otimes v_\mu^\prec)$ is a linear combination with positive coefficients of all the vectors of the form $\chi^{T_1}(v_\la^\prec)\otimes\chi^{T_2}(v_\mu^\prec)$ with $T_1\in\Pi_\la$, $T_2\in\Pi_\mu$ and $T_1+T_2=T$. We see that the annihilator of $v_\la^\prec\otimes v_\mu^\prec$ is precisely $\mathcal J_{\la+\mu}$ and the proposition follows.
\end{proof}

This proposition provides a commutative algebra structure on $\mathcal P^\prec=\bigoplus_\la (L_\la^\prec)^*$ and we have an alternative way of extracting the GT toric degeneration from the representation theory of $\mU^\prec$. Denote \[e_{i_1,\ldots,i_k}^\prec=\chi^{T(i_1,\ldots,i_k)}(v_{\om_k}^\prec)\in L_{\om_k}^\prec\] (i.e. the image of $e_{i_1,\ldots,i_k}$ in the corresponding graded component) and let $(e_{i_1,\ldots,i_k}^\prec)^*$ compose the dual basis in $(L_{\om_k}^\prec)^*$. Recall the toric ideal $J\subset R$ that cuts out the GT toric variety in $\bP$.
\begin{theorem}
There exists a surjection $R\twoheadrightarrow\mathcal P^\prec$ with kernel $J$ mapping $X_{i_1,\ldots,i_k}$ to $(e_{i_1,\ldots,i_k}^\prec)^*$.
\end{theorem} 
\begin{proof}
It is evident that the components $(L_{\om_k}^\prec)^*$ generate $\mathcal P^\prec$. For $\la=(a_1,\ldots,a_{n-1})$ consider \[W_\la^\prec=\Sym^{a_1}(L_{\om_1}^\prec)\otimes\ldots\otimes\Sym^{a_{n-1}}(L_{\om_{n-1}}^\prec)\subset (L_{\omega_1}^\prec)^{\otimes a_1}\otimes\ldots\otimes (L_{\omega_{n-1}}^\prec)^{\otimes a_{n-1}}.\] The space $W_\la^\prec$ is seen to be a $\mU^\prec$-submodule. $W_\la^\prec$ is also dual to $R_\la$ via the chosen bases (similarly to the proof of Proposition~\ref{dualspaces}). We are to show that $L_\la^\prec\subset W_\la^\prec$ is the orthogonal of $J_\la\subset R_\la$.

However, $J_\la$ is the span of binomials \[\prod_{k=1}^{n-1}\prod_{l=1}^{a_k}X_{i^{k,l}_1,\dots,i^{k,l}_k}-\prod_{k=1}^{n-1}\prod_{l=1}^{a_k}X_{j^{k,l}_1,\dots,j^{k,l}_k}\] with \[\sum_{k=1}^{n-1}\sum_{l=1}^{a_k}T(i^{k,l}_1,\dots,i^{k,l}_k)=\sum_{k=1}^{n-1}\sum_{l=1}^{a_k}T(j^{k,l}_1,\dots,j^{k,l}_k)\in\Pi_\la.\] Meanwhile, $\chi^T$ maps the highest weight vector $\bigotimes_k(v_{\om_k}^\prec)^{a_k}$ in $W_\la^\prec$ to (a scalar multiple of) the sum of all \[\bigotimes_{k=1}^{n-1} \prod_{l=1}^{a_k}e_{i^{k,l}_1,\dots,i^{k,l}_k}\] with $\sum_{k,l}T(i^{k,l}_1,\dots,i^{k,l}_k)=T$ (where we refer to the symmetric multiplication in $\Sym(L_{\om_k}^\prec)$). The theorem follows.
\end{proof}

\begin{remark}\label{final}
To conclude let us point out that the constructions in this section appear to have a certain potential for generalization. One could define a total order $\prec'$ on $\{\hat f_{i,j}\}^*$ analogous to $\prec$ but for a different ordering of the generators $\hat f_{i,j}$. If exactly one of $\hat f_{i,l}\prec'\hat f_{i,j}$ and $\hat f_{i,l}\prec'\hat f_{j,l}$ holds for any $i<j<l$, we have a description of $\mU^{\prec'}$ similar to Proposition~\ref{nonzeromons}. This lets us define the sets $\Pi'_\la$ consisting of exponent vectors of monomials acting nontrivially on $L_\la^{\prec'}$. If, in addition, we have $\Pi'_\la+\Pi'_\mu=\Pi'_{\la+\mu}$, then we also obtain analogs of other results in this section. It would be interesting to construct other such examples, especially such where non-abelian gradings are necessary (as they are here).

Another question is whether some generalization of the above setting would let one obtain the degenerate representation theories from Section~\ref{degenact} directly. One would certainly need to invoke gradings by semigroups different from $\{\hat f_{i,j}\}^*$ as discussed in Remark~\ref{notmonomial} (since the annihilating ideals are not monomial). More generally, it would be interesting to adjust the construction in any way that produces the intermediate Gelfand-Tseltin degenerations (i.e. given by points lying on proper faces of $\mathcal K$). 
\end{remark}

\section*{Acknowledgements}

The author would like to thank Evgeny Feigin, Xin Fang, Ievgen Makedonskyi and Oksana Yakimova for helpful discussions of these subjects. The work was partially supported by the grant RSF 19-11-00056. This research was supported in part by the Young Russian Mathematics award.

\end{document}